\documentclass[10pt]{amsart}
\usepackage{amsmath}
\usepackage{amssymb}
\usepackage{enumerate}
\usepackage{amsbsy}
\usepackage{amsfonts}
\usepackage{color}
\usepackage{upgreek}

\newtheorem{thm}{Theorem}[section]
\newtheorem{lem}[thm]{Lemma}
\newtheorem{prop}[thm]{Proposition}

\newtheorem{defn}[thm]{Definition}

\numberwithin{equation}{section}

\begin{document}

	\title[Large data global well-posedness for Dirac equations]{Conditional large-data global well-posedness of Dirac equation with Hartree-type nonlinearity}
	
	\author[Y. Cho]{Yonggeun Cho}
	\address{ Department of Mathematics, and Institute of Pure and Applied Mathematics, Jeonbuk National University, Jeonju 54896, Republic of Korea}
	\email{changocho@jbnu.ac.kr}

    \author[S. Hong]{Seokchang Hong}
    \address{Department of Mathematics, Chung-Ang University, Seoul 06974, Republic of Korea}
    \email{seokchangh11@cau.ac.kr}

 \author[K. Lee]{Kiyeon Lee}
\address{Stochastic Analysis and Application Research Center(SAARC), Korea Advanced Institute
	of Science and Technology, 291 Daehak-ro, Yuseong-gu, Daejeon, 34141, Republic of
	Korea}
\email{kiyeonlee@kaist.ac.kr}

	
	\begin{abstract}
		We study the Cauchy problems for the Hartree-type nonlinear Dirac equations with Yukawa-type potential in two and three spatial dimensions. This paper improves our previous results \cite{chohlee,cholee}; we establish global well-posedness and scattering for large data with a certain condition. Firstly we investigate the long-time behavior of solutions to the Dirac equation satisfies good control provided that a particular dispersive norm of solutions is bounded. The key of our proof relies on modifying multilinear estimates obtained in our previous papers. Secondly, we obtain large data scattering by exploiting the Majorana condition.
	\end{abstract}

		\maketitle

\section{Introduction}

We consider the nonlinear Dirac equation in $\mathbb R^{1+d}$, $d=2,3$:
\begin{align}\label{main-eq}
\left\{
\begin{array}{l}
-i\gamma^\mu\partial_\mu\psi+M\psi = [V_b*(\psi^\dagger \gamma^0\psi)] \psi \\
\psi|_{t=0}=\psi_0,
\end{array}\right.
\end{align}
where $\psi^\dagger$ is the complex conjugate transpose of $\psi$, i.e., $\psi^\dagger = (\psi^T)^*$. The $V_b:=V_b(x)$ is the spatial potential, given by
\begin{align}
V_b(x) = \frac1{4\pi}\frac{e^{-b|x|}}{|x|},\quad x\in\mathbb R^3,
\end{align}
and
\begin{align}
V_b(x) = \int_0^\infty e^{-b^2r-\frac{|x|^2}{4r}}\,\frac{dr}{r} \approx \begin{cases}
 	e^{-b|x|}|bx|^{-\frac12},\quad |x|\gtrsim1, \\
 	-\log|x|,\quad |x|\ll1,
 \end{cases}	x\in\mathbb R^2,
\end{align}
for $b>0$. The unknown spinor field $\psi:\mathbb R^{1+d}\rightarrow \mathbb C^{\tilde d}$ with $(d,\tilde d)=(2,2)$ and $(d,\tilde d)=(3,4)$ is the Dirac field. The matrices $\gamma^\mu$ are the anti-Hermitian $\tilde d\times \tilde d$ Dirac gamma matrices defined as follows: if $\tilde d = 4$, then
$$
\gamma^0 = \begin{bmatrix} \mathbb I_2 & \mathsf 0 \\ \mathsf 0 & -\mathbb I_2  \end{bmatrix},\quad \gamma^j = \begin{bmatrix} \mathsf0 & \sigma^j \\ -\sigma^j & \mathsf0 \end{bmatrix}
$$
where $\mathbb I_2$ is the $2\times2$ identity matrix and $\mathsf 0$ is the zero matrix with obvious dimensions and $\mathsf\sigma^j$ is the Pauli matrices given by
$$
\sigma^1 = \begin{bmatrix} 0 & \;\;1 \\ 1 & \;\;0 \end{bmatrix}, \quad \sigma^2 = \begin{bmatrix} 0 & -i \\ i & \;\;0 \end{bmatrix},\quad \sigma^3 = \begin{bmatrix} 1 &\;\; 0 \\ 0 & -1 \end{bmatrix}
$$
and if $\tilde d = 2$, then $\gamma^0 = \sigma^3$, $\gamma^1 = i\sigma^2$, and $\gamma^2 = -i\sigma^1$.

 The positive constant $M>0$ denotes the mass. We use the summation convention with respect to the Greek indices $\mu=0,1,\dots,d$ and Roman indices $j=1,\cdots,d$. For example, $-i\gamma^\mu\partial_\mu$ means $-i\partial_t+\gamma^j\partial_j$. In practice it is convenient to adapt the $\alpha,\beta$ notation. To be precise let $\alpha^\mu = \gamma^0\gamma^\mu$ and $\beta = \gamma^0$ and multiply our equation \eqref{main-eq} by $\gamma^0$ on the left. Then \eqref{main-eq} is rewritten as
\begin{align}\label{dirac}
\left\{
\begin{array}{l}
(-i\alpha^\mu \partial_\mu + M\beta)\psi = [V_b*(\psi^\dagger\beta\psi)] \beta\psi,\\
\psi|_{t=0}=\psi_0.
\end{array}\right.
\end{align}
For simplicity we always assume that $M = 1$ throughout the paper.

In the study of dispersive property for the Dirac equations, it is accessible to decompose the Dirac spinor field $\psi$ into half waves, i.e., $\psi_+$ and $\psi_-$.
To do this we define projection operators $\Pi_{\theta}$ for $\theta \in \{+, -\}$
by
\[
\Pi_{\theta} :=\frac{1}{2}\left(\mathbb I + \theta \Lambda^{-1}\Big[\alpha^x \cdot (-i\nabla) + \beta\Big]\right),
\]
where $\mathbb I$ is the $\tilde d\times \tilde d$  identity matrix, $\Lambda = \Lambda( -i\nabla ) := \mathcal F^{-1}(\Lambda( \xi ))$, $\Lambda( \xi ) = \sqrt{1 + |\xi|^2}$, and $\alpha^x \cdot (-i\nabla) = \alpha^j(-i\partial_j)$.
 Then $$\Pi_\theta + \Pi_{-\theta} = \mathbb I,\;\; \Pi_\theta \Pi_{-\theta} = \mathsf 0, \;\;\Pi_\theta^2 = \Pi_\theta,$$
and
\[
\Lambda \left(\Pi_{+} - \Pi_- \right) = \alpha^x \cdot(-i\nabla) + \beta 
\] 
  We denote the symbol of $\Pi_\theta$ by $\Pi_\theta(\xi)$. Let $\psi_\theta = \Pi_\theta \psi$. Then we have the following half-Klein-Gordon equation from the equation \eqref{dirac}:
\begin{align}\label{h-w}
\left\{
\begin{array}{l}
(-i\partial_t + \theta\Lambda)\psi_\theta = \Pi_\theta\big([V_b*(\psi^\dagger\beta\psi)] \beta\psi\big),\\
\psi_\theta(0)= \psi_{0, \theta}.
\end{array}\right.
\end{align}
 Note that $\psi =\psi_+ + \psi_-$.

The present authors obtained small data global well-posedness and scattering for the equation \eqref{h-w} for $L_x^{2,\sigma}(\mathbb R^3) (\sigma > 0)$ and $L_x^2(\mathbb R^2)$ data in \cite{cholee} and \cite{chohlee}, respectively. The space $L_x^{2, \sigma}$ is the angularly regular space defined by $\Lambda_{\mathbb S^2}^{-\sigma}L_x^2$ and its norm  is defined by $\|f\|_{L_x^{2, \sigma}} := \|\Lambda_{\mathbb S^2}^\sigma f\|_{L_x^2}$, where $\Lambda_{\mathbb S^2} = (1 - \Delta_{\mathbb S^2})^\frac12$ and $\Delta_{\mathbb S^2}$ is the Laplace-Beltrami operator on the unit sphere $\mathbb S^2$. In this paper we aim to establish them for \textit{conditionally large} data, i.e., we do not assume the smallness of initial data in $L^{2,\sigma}(\mathbb R^3), L^2(\mathbb R^2)$, but pursue to find conditions for solutions and data, which is crucial to the proof the large data scattering.

Recently small data scattering results for the equation \eqref{h-w} have been well-known. For instance see \cite{cholee,chohlee,choozlee,tes,tes1,cyang} and references therein. We also refer to \cite{geosha}, which concerns global solutions for large $H^s$-data in $\mathbb R^{1+2}$. However the large data global-in-time existence at the scaling critical regularity is still open. The main difficulty in proving global existence for large data is that it is not easy to exploit the conservation law such as
$$
\|\psi(t)\|_{L^2_x(\mathbb R^d)} = \|\psi_0\|_{L^2_x(\mathbb R^d)}
$$
due to the dependency of Picard's iteration on the auxiliary spaces like $U^2, V^2$, which is proper subspace of $C([0, T^*); L_x^2)$. Even in the three spatial dimensional settings, we know the existence of global solutions only for $L^{2,\sigma}$-data $(\sigma > 0)$.
\subsection{Bounded dispersive norm condition}

One strategy to overcome such a problem is to enlighten the following question: whether the time evolution of the equation with large data may obey good control under a particular condition to somewhat rough space. It turns out to be the case in the spirit of the work by \cite{candyherr,candyherr1}.
To be precise, given an interval $I\subset\mathbb R$ and $s, \sigma \in \mathbb R$, we define the dispersive type norm in 3d by
\begin{align}
\|u\|_{\mathbf D^{s, \sigma}(I)} = \left( \sum_{N\in 2^\mathbb N}N^{2\sigma}\|\Lambda^{s} H_N u\|_{L^4_tL^4_x(I\times\mathbb R^3)}^2 \right)^\frac12,	
\end{align}
where $H_N$ denotes the projection on angular frequencies of size $N$ (see \eqref{hn} below). When $d = 2$, we do not consider angular regularity and we always put $\sigma = 0$. We define 2d dispersive type norm by
$$
\|u\|_{\mathbf D^{s, 0}(I)} = \|\Lambda^s u\|_{L^4_tL^4_x(I\times\mathbb R^2)}.
$$
We choose $s = -\frac12$ due to the regularity loss of Strichartz bound $\|\psi\|_{\mathbf D^{-\frac12, \sigma}}\lesssim\|\psi(0)\|_{L^{2, \sigma}}$ for free spinor $\psi$ and thus $\mathbf D^{-\frac12, \sigma}$ is a quite rough space. Moreover, it is also meaningful in 3d because the norm $\|\cdot\|_{\mathbf D^{-\frac12, \sigma}}$ becomes a scaling-critical space in view of the $L^2$-scaling critical structure of \eqref{main-eq}. (See \cite{choozxia}.) Our main result clarifies the above strategy. Once we keep the boundedness in a rough space $\mathbf D^{-\frac12,\sigma}$, we can control the time evolution property of the equation for large initial data.
\begin{thm}\label{large-gwp}
Let $\sigma > 0$ for $d = 3$ and $\sigma = 0$ for $d = 2$. We let $\sigma>0$ be arbitrarily small positive number. Consider any maximal $L^{2, \sigma}$-solution
$$
\psi \in C_{\rm loc}\left(I^*;L^{2,\sigma}\left(\mathbb R^d,\mathbb C^{\tilde d}\right)\right),
$$	
to the equation \eqref{main-eq} on $\mathbb R^{1+d}$. Suppose that the solutions satisfy the following boundedness condition:
$$
\sup_{t\in I^*}\|\psi(t)\|_{L^{2, \sigma}(\mathbb R^d)} <+\infty.
$$
If $\|\psi\|_{\mathbf D^{-\frac12, \sigma}(I^*)}<\infty$, then we have $I^*=\mathbb R$ and the solution $\psi$ scatters to a free solution in $L_x^{2, \sigma}$ as $t\rightarrow\pm\infty$.
\end{thm}
The proof follows from a refinement of the multilinear estimates by the previous works \cite{chohlee,cholee}. That is, we make a small quantity for $L^4_{t,x}$ norms on the right-hand side of trilinear estimates. The trilinear estimates in the low modulation regime are rather complicated than in the high modulation regime. When the modulation is lower than the highest input-frequency, we study bilinear forms $\varphi^\dagger\beta\phi$ and decompose it into the modulation localised terms. We deal with all possible frequency interactions such as High$\times$High to Low and Low$\times$High cases. The most delicate case in the proof occurs when the modulation is between the highest frequency and lowest frequency. Since we are only allowed to use $L^4_{t,x}$-Strichartz estimates in our approach, it gives rise to a slightly bigger bound than in other cases.
\subsection{Majorana condition}
Now we shall consider the problem of time evolution property of large dispersive solutions of the equation \eqref{main-eq} in the different way. Let us assume that the initial data $\psi_0$ satisfies
%
\begin{align}\label{majorana-condi}
\psi_0+z\gamma^2\psi^*_0=0	
\end{align}
with $z=e^{i\omega}, \omega \in \mathbb R$. Here $z\gamma^2\psi^*$ is referred as a charge conjugation of Dirac field $\psi$. \eqref{majorana-condi} implies the Dirac field is its own antimatter field \cite{bj-dr}. Under this condition the cubic Dirac equations reduce to linear equation of $\psi$, since we have $\overline\psi\psi=0$, where $\overline\psi=\psi^\dagger\gamma^0$. For this see \cite{chagla, ozya}. Such a structural condition is first introduced by Ettore Majorana \cite{majorana}, and we refer this condition to the Majorana condition. We can readily obtain global existence and scattering for the equation \eqref{main-eq} within the Majorana condition for any initial data, since the corresponding solution will evolve \textit{linearly} in time.

The natural question is then whether the time evolution property may be stable by a small perturbation on the Majorana condition for a given large initial data. That is, we do not assume that $\psi_0+z\gamma^2\psi_0^*$ is not identically zero. To be precise, for any large initial data $\psi_0$ such that $\|\psi_0\|_{L^{2,\sigma}} = {\mathsf A} \gtrsim 1$, we make some perturbation on \eqref{majorana-condi} as follows:
\begin{align}
\|\psi_0+z\gamma^2\psi_0^*\|_{L^{2,\sigma}}\le\mathtt a,	
\end{align}
    for a relatively small $\mathtt a > 0$ which depends on the quantity ${\mathsf A}$. Thus we aim to prove the following theorem.
    \begin{thm}\label{majorana-gwp}
Let $z=e^{i\omega}, \omega \in \mathbb R$. Let $\sigma > 0$ for $d= 3$ and $\sigma = 0$ for $d= 2$. For any ${\mathsf A} > 0$, there exists $\mathtt a = \mathtt a({\mathsf A}) > 0$ such that for all initial data satisfying
$$
\|\psi_0\|_{L^{2,\sigma}(\mathbb R^d)}\le {\mathsf A} \textrm{ and } \|\psi_0+z\gamma^2\psi_0^*\|_{L^{2,\sigma}(\mathbb R^d)} \le \mathtt a
$$
the Cauchy problem of \eqref{main-eq} on $\mathbb R^{1+d}$ is globally well-posed and solutions scatter to free solutions in $L^{2, \sigma}$ as $t\rightarrow\pm\infty$.
\end{thm}
  \noindent  However, the smallness condition on the form $\psi_0+z\gamma^2\psi_0^*$ yields large norm on the other side as follow:
    \begin{align}
    \|\psi_0 - z\gamma^2\psi_0^*\|_{L^{2,\sigma}}\le \mathsf {\mathsf A}.	
    \end{align}
Now we decompose the spinor field $\psi$ into
\begin{align}
\psi = \frac12(\psi+z\gamma^2\psi^*)+\frac12(\psi-z\gamma^2\psi^*)=: P_+^z\psi +P^z_-\psi. 	
\end{align}
Simple observation leads us that for $\theta\in\{+,-\}$, $P^z_\theta P^z_\theta=P^z_\theta$ and $P^z_\theta P^z_{-\theta}=0$ and hence the operator $P^z_\theta$ is truly the projection operator. Thanks to the perturbed Majorana condition, we do not need to control any rough space during iteration process as what we have done in Theorem \ref{large-gwp}. Instead, we will show that the time evolution property of $\psi(t)$ for the data $\psi_0$ is determined by the two evolutions; $P^z_\theta\psi$ and $P_{-\theta}^z\psi$. More precisely, instead of studying the Cauchy problems for \eqref{main-eq}, we are concerned with the Cauchy problems for a system of cubic Dirac equations as follows: for sufficiently smooth $\varphi, \phi$, 
\begin{align}
\begin{aligned}
-i\gamma^\mu\partial_\mu\varphi+\varphi = V_b*(\overline{P^z_\theta\varphi}P^z_{-\theta}\phi+\overline{P^z_{-\theta}\phi}P^z_\theta\varphi)\varphi, \\
-i\gamma^\mu\partial_\mu\phi+\phi = V_b*(\overline{P^z_\theta\varphi}P^z_{-\theta}\phi+\overline{P^z_{-\theta}\phi}P^z_\theta\varphi)\phi.
\end{aligned}
\end{align}
with initial data
\begin{align}
\varphi|_{t=0} = P^z_\theta\psi_0,\quad \phi|_{t=0} = P^z_{-\theta}\psi_0.	
\end{align}
Recall that $\overline\psi=\psi^\dagger\gamma^0$. By taking $P^z_\theta$ to both sides for the first equation and $P^z_{-\theta}$ for the second equation respectively, we would obtain solutions $P^z_\theta\varphi(t)$ and $P^z_{-\theta}\phi(t)$ to the system. Now we denote
$$
\psi(t) := P^z_\theta\varphi(t) + P^z_{-\theta}\phi(t).
$$
Then $\psi|_{t=0}=\psi_0$ and $\psi$ is truly the solution to the original equation. 

The main scheme of exploiting the Majorana condition is as follows. Given a large data $\psi_0$, say, $\|\psi_0\|_{L^{2,\sigma}}={\mathsf A} \gtrsim 1$, we decompose the initial data with respect to the charge conjugation using the projections $P^z_\theta$. If the most of the norm is occupied in the one of the charge conjugation, say $P^z_{-\theta}\psi_0$, then the other one $P^z_\theta\psi_0$ must have small norm. Then via the standard iteration methods, we see that the time evolution of $P^z_\theta\psi_0$ and $P^z_{-\theta}\psi_0$ would be controlled by the norm $\|P^z_\theta\psi_0\|_{L^{2,\sigma}}$ and $\|P^z_{-\theta}\psi_0\|_{L^{2,\sigma}}$, respectively. In summary, the global-in-time evolution of  $P^z_\theta\psi_0$ and $P^z_{-\theta}\psi_0$ is established in an open neighborhood of large data with critical regularity assumption.

\subsection*{Organisation} We organise the rest of this paper as follow. We give the proof of Theorem \ref{large-gwp} in Section \ref{sketch}. The proof relies heavily on the multilinear estimates Proposition \ref{main-tri-est}. We present preliminaries which will be used for the proof of Proposition \ref{main-tri-est} in Section \ref{pre}. Then Section \ref{sec-multi-3d} and Section \ref{multi-2d} are devoted to the proof of Proposition \ref{main-tri-est} in $d=3$ and $d=2$, respectively. Finally we prove Theorem \ref{majorana-gwp} in Section \ref{sec-maj}.

\subsection*{Notations}
\begin{enumerate}
\item
As usual different positive constants, which are independent of dyadic numbers $\mu,\lambda$, and $h$ are denoted by the same letter $C$, if not specified. The inequalities $A \lesssim B$ and $A \gtrsim B$ means that $A \le CB$ and
$A \ge C^{-1}B$, respectively for some $C>0$. By the notation $A \approx B$ we mean that $A \lesssim B$ and $A \gtrsim B$, i.e., $\frac1CB \le A\le CB $ for some absolute constant $C$. We also use the notation $A\ll B$ if $A\le \frac1CB$ for some large constant $C$. Thus for quantities $A$ and $B$, we can consider three cases: $A\approx B$, $A\ll B$ and $A\gg B$. In fact, $A\lesssim B$ means that $A\approx B$ or $A\ll B$.

The spatial and space-time Fourier transform are defined by
$$
\widehat{f}(\xi) = \int_{\mathbb R^d} e^{-ix\cdot\xi}f(x)\,dx, \quad \widetilde{u}(\tau,\xi) = \int_{\mathbb R^{1+d}}e^{-i(t\tau+x\cdot\xi)}u(t,x)\,dtdx.
$$
We also write $\mathcal F_x(f)=\widehat{f}$ and $\mathcal F_{t, x}(u)=\widetilde{u}$. We denote the backward and forward wave propagation of a function $f$ on $\mathbb R^d$ by
$$
e^{-\theta it \Lambda}f = \frac1{(2\pi)^d}\int_{\mathbb R^d}e^{ix\cdot\xi}e^{-\theta it\Lambda(\xi)}\widehat{f}(\xi)\,d\xi.
$$

\item
We fix a smooth function $\rho\in C^\infty_0(\mathbb R)$ such that $\rho$ is supported in the set $\{ \frac12<r<2\}$ and we let
$$
\sum_{\lambda\in2^{\mathbb Z}}\rho\left(\frac r\lambda\right) =1,
$$
and write $\rho_1=\sum_{\lambda\le1}\rho(\frac r\lambda)$ with $\rho_1(0)=1$.  Now we define the standard Littlewood-Paley multipliers for $\lambda\in 2^{\mathbb N}$ and $\lambda>1$:
$$
P_\lambda = \rho\left(\frac{|-i\nabla|}{\lambda}\right),\quad P_1=\rho_1(|-i\nabla|).
$$

\item
We let $Y_{\ell}$ be the set of homogeneous harmonic polynomial of degree $\ell$ on $\mathbb R^3$. Then define $\{ y_{\ell,n} \}_{n=0}^{2\ell}$ a set of orthonormal basis for $Y_{\ell}$, with respect to the inner product:
\begin{align}
\langle y_{\ell,n},y_{\ell',n'}\rangle_{L^2_\omega(\mathbb S^2)} = \int_{\mathbb S^2}{y_{\ell,n}(\omega)} \overline{y_{\ell',n'}(\omega)}\,d\omega.
\end{align}
Given $f\in L^2_x(\mathbb R^3)$, we have the orthogonal decomposition as follow:
\begin{align}
f(x) = \sum_{\ell}\sum_{n=0}^{2\ell}\langle f(|x|\omega),y_{\ell,n}(\omega)\rangle_{L^2_\omega(\mathbb S^2)}y_{\ell,n}\big(\frac{x}{|x|}\big).
\end{align}
For a dyadic number $N>1$, we define the spherical Littlewood-Paley decompositions by
\begin{align}\begin{aligned}\label{hn}
H_N(f)(x) & = 	\sum_{\ell}\sum_{n=0}^{2\ell}\rho\left(\frac\ell N\right)\langle f(|x|\omega),y_{\ell,n}(\omega)\rangle_{L^2_\omega(\mathbb S^2)}y_{\ell,n}\big(\frac{x}{|x|}\big), \\
H_1(f)(x) & = \sum_{\ell}\sum_{n=0}^{2\ell}\rho_{\le1}(\ell)\langle f(|x|\omega),y_{\ell,n}(\omega)\rangle_{L^2_\omega(\mathbb S^2)}y_{\ell,n}\big(\frac{x}{|x|}\big).
\end{aligned}\end{align}
\end{enumerate}

\section{Proof of Theorem \ref{large-gwp}}\label{sketch}
We define our main function norm as
\begin{align*}
\|\phi\|_{F^{s,\sigma}_\theta(I)} &= \left( \sum_{\lambda,N\in 2^\mathbb N}\lambda^{2s}N^{2\sigma}\|\Pi_\theta P_\lambda H_N \phi\|_{V^2_\theta(I)}^2 \right)^\frac12, \\
\|\phi\|_{F^{s,0}_\theta(I)} &= \left( \sum_{\lambda\in 2^\mathbb N}\lambda^{2s}\|\Pi_\theta P_\lambda \phi\|_{V^2_\theta(I)}^2 \right)^\frac12,
\end{align*}
where $\theta \in \{+, -\}$. If $I = \mathbb R$, then we drop $I$. We say $\psi \in F^{s, \sigma}(I)$ if $\phi_\theta \in F_\theta^{s, \sigma}(I)$ for each $\theta \in \{+, -\}$. We write $$\|\phi\|_{F^{s, \sigma}(I)} = \sum_{\theta \in \{+,-\}}\|\phi_\theta\|_{F_\theta^{s, \sigma}(I)}.$$

\subsection{Local theory}
At the first step we  present the local well-posedness result. To do so we give the explicit definition on the maximal solutions.
\begin{defn}
Let $s,\sigma\in\mathbb R$. If $d = 3$, then $\sigma > 0$, and if $d= 2$, then $\sigma = 0$. When $d = 3$, we define the angularly regular Sobolev spaces by $H^{s, \sigma} := \{f \in H^s : \|f\|_{H^{s, \sigma}} := \|\Lambda^s\Lambda_{\mathbb S^2}^\sigma f\|_{L_x^2} < \infty\}$.
\begin{enumerate}
\item We say $\psi:I\times\mathbb R^d\rightarrow\mathbb C^{\tilde d}$ is an $H^{s, \sigma}$-strong solution on an interval $I\subset\mathbb R$ if
$$
\psi\in C(I,H^{s, \sigma}(\mathbb R^d,\mathbb C^{\tilde d})
$$	
and there exists a sequence $\psi_n\in C^2(I,H^m(\mathbb R^d,\mathbb C^{\tilde d}))$, $m=\max\{10,s\}$, of classical solutions to \eqref{main-eq} such that for any compact $I'\subset I$,
$$
\sup_{t\in I'}\|\psi(t)-\psi_n(t)\|_{H^{s, \sigma}(\mathbb R^d,\mathbb C^{\tilde d})} \rightarrow 0,
$$
as $n\rightarrow\infty$.
\item We say $\psi:[t_0,t^*)\times\mathbb R^d\rightarrow\mathbb C^{\tilde d}$ is a (forward) maximal $H^{s, \sigma}$-solution if the following two properties hold:
\begin{enumerate}
\item for any $t_1\in(t_0,t^*)$, $\psi$ is a strong $H^{s, \sigma}$-solution on $[t_0,t_1)$;
\item if $\psi':I'\times\mathbb R^d\rightarrow\mathbb C^{\tilde d}$ is a strong $H^{s, \sigma}$-solution on an interval $I$ satisfying $I\cap[t_0,t^*)\neq\varnothing$ and $\psi'=\psi$ on $I\cap[t_0,t^*)$, then $I\cap[t_0,\infty)\subset[t_0,t^*)$.
\end{enumerate}
\end{enumerate}
\end{defn}
Given $t_0\in I\subset\mathbb R$ and $F\in L^\infty_tL^2_x(I\times\mathbb R^d)$, for $t\in I$ and $\theta \in \{+, -\}$ let $\mathcal I_{\theta}[F](t_0; t)$ denote the inhomogeneous solution operator for the half-Klein-Gordon equation $(-i\partial_t + \theta \Lambda)\phi = \Pi_\theta F, \phi(t_0) = 0$. Then
$$
\mathcal I_\theta [F](t_0; t) = i\int_{t_0}^t e^{-\theta i(t-t')\Lambda}\Pi_\theta [F](t')\,dt'.
$$
The solution to \eqref{h-w} is written as
$$
\psi_\theta(t) = e^{-\theta i(t-t_0)\Lambda}\psi_\theta(t_0) + \mathcal I_\theta[f](t_0; t),
$$
where $f = [V_b*(\psi^\dagger\beta\psi)] \beta\psi$. We shall prove the following trilinear estimates in the next two sections.
\begin{prop}\label{main-tri-est}
Let $\sigma > 0$ when $d = 3$ and $\sigma = 0$ when $d = 2$. There exist constants $C > 0$ and $\eta<\frac\sigma{10}$ such that if $I\subset\mathbb R$ is a left-closed interval, $t_0\in I$, $\varphi,\phi,\psi\in F^{0,\sigma}(I)$, then for any $\theta \in \{+, -\}$, we have the bounds
\begin{align}\label{0-tri}
\begin{aligned}
	\|\mathcal I_\theta [V_b*(\varphi^\dagger\beta\phi)\beta\psi](t_0; \cdot)\|_{F^{0,\sigma}(I)} & \le C (\|\varphi\|_{\mathbf D^{-\frac12, \sigma}(I)}\|\phi\|_{\mathbf D^{-\frac12, \sigma}(I)}\|\psi\|_{\mathbf D^{-\frac12, \sigma}(I)})^\eta \\
	&\qquad\qquad \times (\|\varphi\|_{F^{0,\sigma}(I)}\|\phi\|_{F^{0,\sigma}(I)}\|\psi\|_{F^{0,\sigma}(I)})^{1-\eta},
	\end{aligned}
\end{align}	
and for any $s > 0$, we have the fractional Leibniz type bounds
\begin{align}\label{frac-tri}
\begin{aligned}
	\|\mathcal I_{\theta}[V_b*(\varphi^\dagger\beta\phi)\beta\psi](t_0; \cdot)\|_{F^{s,\sigma}(I)} & \le 2^sC\big(\|\varphi\|_{\mathbf D^{-\frac12, \sigma}(I)}^\eta \|\varphi\|_{F^{0,\sigma}(I)}^{1-\eta}\|\phi\|_{F^{s,\sigma}(I)}\|\psi\|_{F^{s,\sigma}(I)} \\
	& \quad + \|\varphi\|_{F^{s,\sigma}(I)}\|\phi\|_{\mathbf D^{-\frac12, \sigma}(I)}^\eta\|\phi\|_{F^{0,\sigma}(I)}^{1-\eta}\|\psi\|_{F^{s,\sigma}(I)} \\
	& \quad + \|\varphi\|_{F^{s,\sigma}(I)}\|\phi\|_{F^{s,\sigma}(I)}\|\psi\|_{\mathbf D^{-\frac12,\sigma}(I)}^\eta \|\psi\|_{F^{0,\sigma}(I)}^{1-\eta}\big).
	\end{aligned}
\end{align}
\end{prop}
\noindent In fact the multilinear estimates \eqref{frac-tri} follow from the estimate \eqref{0-tri}, and hence we only give the proof of \eqref{0-tri} in this paper.
By an application of Proposition \ref{main-tri-est} we are able to prove local well-posedness result.
\begin{prop}[Local well-posedness]\label{lwp}
	Let $\sigma > 0$ for $d = 3$ and $\sigma =0$ for $d= 2$. Then there exist $0<\eta<\frac\sigma{10}$ and $C>1$, such that if
	$$
	{\mathsf A}>0,\quad 0 < \mathtt a < (C{\mathsf A}^{1-\eta})^{-\frac1\eta},
	$$
	and $I\subset\mathbb R$ is a left-closed interval, then for any initial time $t_0\in I$, and any data $\psi_0\in L^{2, \sigma}$ satisfying
	$$
	\|\psi_0\|_{L^{2, \sigma}} < {\mathsf A},\quad \sum_{\theta \in \{+, -\}}\|e^{-\theta i(t-t_0)\Lambda}\psi_{0, \theta}\|_{\mathbf D^{-\frac12, \sigma}(I)} < \mathtt a,
	$$
	there exists a unique $L^{2, \sigma}$-strong solution $\psi$ of \eqref{main-eq} on $I$ with $\psi(t_0)=\psi_0$. Moreover, the data-to-solution map is Lipschitz-continuous into $F^{0,\sigma}(I)$ and we have the bounds
	$$
	\bigg\|\psi - \sum_{\theta \in \{+, -\}}e^{-\theta i(t-t_0)\Lambda}\psi_{0, \theta}\bigg\|_{F^{0,\sigma}(I)} \le C({\mathtt a}^\eta {\mathsf A}^{1-\eta})^3.
	$$
\end{prop}
\begin{proof}[Proof of Proposition \ref{lwp}]
Let us set
$$
\epsilon_0 = \mathtt a^\eta {\mathsf A}^{1-\eta}.
$$
For simplicity, we write
$$	
\psi_{\ell} := \sum_\theta e^{-\theta i(t-t_0)\Lambda}\psi_{0, \theta},\quad \psi_{\mathcal N} := \psi - \psi_\ell.
$$
We define the set $\mathcal X\subset F^{0,\sigma}(I)$ as
$$
\mathcal X:=\{ \psi\in F^{0,\sigma}(I) : \|\psi_\mathcal N\|_{F^{0,\sigma}(I)} \le \epsilon_0\}.
$$
We aim to construct a fixed point of the map $\Phi:\mathcal X\rightarrow\mathcal X$ defined as
$$
\Phi(\psi) = \psi_\ell + \sum_\theta\mathcal I_{\theta}[V_b*(\psi^\dagger\beta\psi)\beta\psi](t_0; \cdot).
$$
To do this, let $\psi\in\mathcal X$. Then after decomposing the product
\begin{align*}
(\varphi^\dagger\phi)\psi & = \varphi^\dagger_\ell\phi_\ell\psi_\ell+\varphi_\mathcal N^\dagger \phi_\ell\psi_\ell +\varphi_\ell^\dagger\phi_\mathcal N\psi_\ell+\varphi_\mathcal N^\dagger\phi_\mathcal N\psi_\ell \\
&\qquad + \varphi^\dagger_\ell\phi_\ell\psi_\mathcal N+\varphi_\mathcal N^\dagger \phi_\ell\psi_\mathcal N +\varphi_\ell^\dagger\phi_\mathcal N\psi_\mathcal N+\varphi_\mathcal N^\dagger\phi_\mathcal N\psi_\mathcal N
\end{align*}
by an application of Proposition \ref{main-tri-est} together with bounds
\begin{align}
\|\psi_{\ell}\|_{F^{0,\sigma}(I)} & \lesssim \|\psi_0\|_{L^{2, \sigma}},\label{lin-est} \\
\|\psi\|_{L^\infty_tL^{2,\sigma}(I\times\mathbb R^3)} + \|\psi\|_{\mathbf D^{-\frac12, \sigma}} & \lesssim \|\psi\|_{F^{0,\sigma}(I)},	\label{energy-str-est}
\end{align}
we see that there exists $C^*>0$ such that
\begin{align}
\|\mathcal I_\theta [V_b*(\varphi^\dagger\beta\phi)\beta\psi](t_0; \cdot)\|_{F^{0,\sigma}(I)} \le C^* (\epsilon_0)^3.	
\end{align}

To show that $\Phi$ is a contraction on $\mathcal X$, for $\psi,\psi'\in\mathcal X$, we apply Proposition \ref{main-tri-est} and \eqref{lin-est}, \eqref{energy-str-est} to obtain
\begin{align*}
&\|\mathcal I_{\theta}[V_b*(\varphi^\dagger\beta\phi)\beta\psi](t_0; \cdot) - \mathcal I_{\theta}[V_b*(\varphi'^\dagger\beta\phi')\beta\psi'](t_0; \cdot)\|_{F^{0,\sigma}(I)}	\\
& \le C^* (\epsilon_0)^2(\|\varphi-\varphi'\|_{F^{0,\sigma}(I)}+\|\psi-\psi'\|_{F^{0,\sigma}(I)}+\|\phi-\phi'\|_{F^{0,\sigma}(I)}).
\end{align*}
Then we set
$$
3C^*(\epsilon_0)^2 \le \frac12.
$$
That is,
$$
\mathtt a^\eta {\mathsf A}^{1-\eta} \le (6C^*)^{-\frac12}.
$$
We conclude that the standard contraction principle implies that there exists a unique fixed point in $\mathcal X$, and the solution map $\Phi$ depends continuously on the initial datum. Now we set $\mathcal X_s\subset F^{s,\sigma}$ as
$$
\mathcal X_s := \{ \psi\in F^{s,\sigma}(I) : \|\psi\|_{F^{s,\sigma}(I)}\le \epsilon_s \},
$$
where
$$
\epsilon_s = \|\psi_0\|_{H^s(\mathbb R^d)}.
$$
By the multilinear estimates \eqref{frac-tri}, we have
\begin{align}
	\|\mathcal I_{\theta}[V_b*(\varphi^\dagger\beta\phi)\beta\psi](t_0; \cdot)\|_{F^{s,\sigma}(I)} \le C^*(s)\epsilon_0^2\epsilon_s.
\end{align}
We repeat the similar argument to obtain a unique fixed point in $\mathcal X_s$ and then fix $$C=\sup_{0\le s\le10}C^*(s).$$  Suppose that $0<\mathtt a<(C\mathsf A^{1-\eta})^{-\frac1\eta}$. Then we get a unique solution $\psi\in\mathcal X$, which depends continuously on the data. Now we approximate the data with functions in $H^{10,\sigma}$ and applying the above argument with $s=10$, we obtain a sequence of solutions in $\mathcal X_{10}$, which converge to $\psi$. Thus $\psi$ is an $L^{2, \sigma}$-strong solution to \eqref{main-eq}.
\end{proof}
\subsection{Conditional regularity}
We shall prove that $L^{2, \sigma}$-solutions belong to $F^{0,\sigma}(I)$ provided that the $L^4_{t,x}$-norm is sufficiently small relative to the data norm $\|\psi(t_0)\|_{L^{2, \sigma}}$.
\begin{thm}\label{data-dis-thm}
Let $\sigma>0$ for $d = 3$ and $\sigma = 0$ for $d = 2$. Then there exists $0<\eta<\frac\sigma{10}$ and $C>1$ such that if ${\mathsf A}>0$ and $I\subset\mathbb R$ is a left-closed interval, $t_0\in I$, and $\psi$ is an $L^{2, \sigma}$-strong solution on $I$ satisfying
$$
\|\psi(t_0)\|_{L^{2, \sigma}} \le {\mathsf A}
$$	
and
$$
\|\psi\|_{\mathbf D^{-\frac12, \sigma}(I)} \le (C(1+{\mathsf A})^{3-\eta})^{-\frac1\eta},
$$
then $\psi\in F^{0,\sigma}(I)$ and we have the bound
$$
\|\psi\|_{F^{0,\sigma}(I)} \le CA.
$$
\end{thm}
\begin{proof}[Proof of Theorem \ref{data-dis-thm}]
We first consider $I=[t_0,t_1)$ with $t_1\le\infty$. Let $\psi$ be an $L^{2, \sigma}$-solution on $I$ and we define
$$
\delta := \|\psi\|_{\mathbf D^{-\frac12, \sigma}}
$$	
and
$$
\mathbf T := \left\{ t_0<T\le t_1 : \sup_{t_0<T'\le T}\|\psi\|_{F^{0,\sigma}([t_0,T'))} \le 2C^*{\mathsf A} \right\}.
$$
The local well-posedness result in Proposition \ref{lwp} implies that $T\in\mathbf T$ provided that $T-t_0$ is sufficiently small, i.e., $\mathbf T \neq\varnothing$. We let
$$
T_{\sup} := \sup\mathbf T.
$$
Our goal is to show that $T_{\sup} = t_1$. We assume that $T_{\sup}<t_1$ and let $T_n\in\mathbf T$ be a sequence converging to $T_{\sup}$. The continuity of the solution $\psi$ at $T_{\sup}$, together with \eqref{energy-str-est} and the definition of $\mathbf T$ implies that
\begin{align*}
\|\psi(T_{\sup})\|_{L^{2, \sigma}} & \le C^* \sup_{t_0<T<T_{\sup}}\|\psi\|_{F^{0,\sigma}([t_0,T))}	\\
& \le (2C^*)^2{\mathsf A}
\end{align*}
Applying Proposition \ref{lwp} again, there exists $n$ and $\epsilon_0>0$ such that for all $0<\epsilon<\epsilon_0$ we have on the interval $[T_n,T_{\sup}+\epsilon)$ the bound
\begin{align*}
\|\psi\|_{F^{0,\sigma}([T_n,T_{\sup}+\epsilon))} & \le 2C^*\|\psi(T_{\sup})\|_{L^{2, \sigma}} \\
& \le (2C^*)^3{\mathsf A}.	
\end{align*}

We now exploit the smallness assumption on the $L^4_{t,x}$ norm. An application of \eqref{energy-str-est}, triangular inequality with respect to the time intervals, together with the trilinear estimates in Proposition \ref{main-tri-est}, and the fact that $\psi$ is an $L^{2, \sigma}$-strong solution on $[t_0,t_1)$ implies that
\begin{align*}
\|\psi\|_{F^{0,\sigma}([t_0,T_{\sup}+\epsilon))} & \le C^*\|\psi(t_0)\|_{L^{2, \sigma}} \\
& + C^*\|\psi\|_{\mathbf D^{-\frac12, \sigma}([t_0,T_{\sup}+\epsilon))}^\eta (\|\psi\|_{F^{0,\sigma}([t_0,T_n))}+\|\psi\|_{F^{0,\sigma}([T_n,T_{\sup}+\epsilon))})^{1-\eta} 	\\
& \qquad\qquad\times ( \|\psi\|_{F^{0,\sigma}([t_0,T_n))}+\|\psi\|_{F^{0,\sigma}([T_n,T_{\sup}+\epsilon))} )^2 \\
& \le C^*{\mathsf A} + \delta^\eta (2C^*)^6{\mathsf A}^2{\mathsf A}^{1-\eta}.
\end{align*}
If we take
$$
\delta \le [ (2C^*)^5 {\mathsf A}^{3-\eta} ]^{-\frac1\eta},
$$
then
$$
\|\psi\|_{F^{0,\sigma}([t_0,T_{\sup}+\epsilon))} \le 2C^*{\mathsf A}.
$$
Consequently, we have $T_{\sup}+\epsilon\in\mathbf T$, which contradicts the assumption $T_{\sup}<t_1$. Therefore, we must have $T_{\sup}=t_1$.
\end{proof}
\subsection{Proof of global well-posedness and scattering}
Now we give the proof of Theorem \ref{large-gwp}. We only consider the forward-in-time problem for the system \eqref{main-eq}, since our main system is time reversible. We let $\psi:[t_0,t^*)\times\mathbb R^d\rightarrow\mathbb C^{\tilde d}$ be a forward maximal $L^{2, \sigma}$-solution to \eqref{main-eq} such that
$$
\sup_{t_0\le t<t^*}\|\psi(t)\|_{L^{2, \sigma}} \le {\mathsf A},\quad \|\psi\|_{\mathbf D^{-\frac12, \sigma}([t_0,t^*))} <\infty.
$$
Since the dispersive norm $\|\cdot\|_{\mathbf D^{-\frac12, \sigma}}$ is finite, by the dominated convergence theorem, we see that for every $\delta>0$, there exists an interval $I=[t_1,t^*)$ with $t_1<t^*$ such that
$$
\|\psi\|_{\mathbf D^{-\frac12, \sigma}(I)} \le\delta.
$$
In particular, choosing $\delta$ sufficiently small, depending only on ${\mathsf A}$, an application of Theorem \ref{data-dis-thm} implies that $\psi\in F^{0,\sigma}(I)$. Therefore, by the existence of left limits in $V^2$, there exists $\psi^\infty \in L^{2, \sigma}$ such that
$$
\lim_{t\rightarrow t^*}\|\psi(t) - \psi_\ell^\infty \|_{L^{2, \sigma}} = 0,
$$
where $\psi_\ell^\infty = \sum_{\theta \in \{+, -\}}e^{-\theta i(t-t_0)\Lambda}\Pi_\theta\psi^\infty$.
Then the local theory, together with the definition of maximal $L^{2, \sigma}$-solution implies  $t^*=\infty$. In consequence, the solution $\psi$ exists globally in time and scatters as $t\rightarrow\infty$.

\section{Preliminaries for Proposition \ref{main-tri-est}}\label{pre}
This section is devoted to introducing preliminaries for Proposition \ref{main-tri-est}, which play a crucial role in the proof of our main result.
\subsection{Multipliers}\label{multi}
We define $\mathcal Q_\mu$ to be a finitely overlapping collection of cubes of diameter $\frac{\mu}{1000}$ covering $\mathbb R^d$, and let $\{ \rho_{\mathsf q}\}_{\mathsf q\in\mathcal Q_\mu}$ be a corresponding subordinate partition of unity.
For $\mathsf q\in\mathcal Q_\mu$, $h\in 2^{\mathbb Z}$ let
$$
P_{\mathsf q} = \rho_{\mathsf q}(-i\nabla),\quad C^{\theta}_h = \rho\left(\frac{|-i\partial_t + \theta\Lambda|}{h}\right).
$$

We define $C^\theta_{\le h}=\sum_{\delta\le h}C^\theta_\delta$ and $C^\theta_{\ge h}$ is defined in the similar way. For simplicity we also write $C^+_h = C_h$.
Given $0 < \mathtt r \lesssim1$, we define $\mathcal C_{\mathtt r}$ to be a collection of finitely overlapping caps of radius ${\mathtt r}$ on the sphere $\mathbb S^2$. If $\kappa\in\mathcal C_{\mathtt r}$, we let $\omega_\kappa$ be the centre of the cap $\kappa$. Then we define $\{\rho_\kappa\}_{\kappa\in\mathcal C_{\mathtt r}}$ to be a smooth partition of unity subordinate to the conic sectors $\{ \xi\neq0 , \frac{\xi}{|\xi|}\in\kappa \}$ and denote the angular localisation Fourier multipliers by
$
R_\kappa = \rho_\kappa(-i\nabla).
$

\subsection{Analysis on the sphere}\label{an-sph}
We introduce some basic facts from harmonic analysis on the unit sphere. The most of ingredients can be found in \cite{candyherr, ster}. We also refer the readers to \cite{steinweiss} for more systematic introduction to the spherical harmonics.
Since $-\Delta_{\mathbb S^2}y_{\ell, n} = \ell(\ell+1)y_{\ell, n}$, by orthogonality one can readily get
$$\|\Lambda_{\mathbb S^2}^\sigma f\|_{L^2_\omega({\mathbb S^2})} \approx \left\|\sum_{N\in2^{\mathbb N}\cup\{0\}}N^\sigma H_Nf\right\|_{L^2_\omega({\mathbb S^2})}.$$

\begin{lem}[Lemma 7.1 of \cite{candyherr}]\label{sph-ortho}
Let $N\ge1$. Then $H_N$ is uniformly bounded on $L^p(\mathbb R^3)$ in $N$, and $H_N$ commutes with all radial Fourier multipliers. Moreover, if $N'\ge1$, then either $N\approx N'$ or
$$
H_N\Pi_\theta H_{N'}=0,
$$	
where
$
\Pi_{\theta} :=\frac{1}{2}\left(\mathbb I + \theta \Lambda^{-1}\Big[\alpha^x \cdot (-i\nabla) + \beta\Big]\right),
$
with $\theta\in\{+,-\}$.
\end{lem}

\subsection{Adapted function spaces}\label{ftn-sp}
Let $\mathcal Z =\left\{ \{t_k\}_{k=0}^K : t_k\in\mathbb R, t_k<t_{k+1} \right\}$ be the set of increasing sequences of real numbers.
We define the $2$-variation of $v$ to be
$$
|v|_{V^2} = \sup_{ \{t_k\}_{k=0}^K\in\mathcal I } \left( \sum_{k=0}^K\|v(t_k)-v(t_{k-1})\|_{L^2_x}^2 \right)^\frac12
$$
Then the Banach space $V^2$ can be defined to be all right continuous functions $v:\mathbb R\rightarrow L^2_x$ such that the quantity
$$
\|v\|_{V^2} = \|v\|_{L^\infty_tL^2_x} + |v|_{V^2}
$$
is finite. Set $\|u\|_{V^2_\theta}=\|e^{-\theta it\Lambda}u\|_{V^2}$. We recall basic properties of $V^2_\theta$ space from \cite{candyherr, candyherr1, haheko}. In particular,  we use the following lemma to prove the scattering result.
\begin{lem}[Lemma 7.4 of \cite{candyherr}]\label{v-scatter}
	Let $u\in V^2_\theta$. Then there exists $f\in L^2_x$ such that $\|u(t)-e^{-\theta it\Lambda}f\|_{L^2_x}\rightarrow0$ as $t\rightarrow\pm\infty$.
\end{lem}

The following lemma is on a simple bound in the high-modulation region.
\begin{lem}[Corollary 2.18 of \cite{haheko}]
Let $2\le q\le\infty$. For $h\in2^{\mathbb Z}$ and $\theta \in \{+, -\}$, we have
\begin{align}\label{bdd-high-mod}
\begin{aligned}
\|C^{\theta}_h u\|_{L^q_tL^2_x} \lesssim h^{-\frac1q}\|u\|_{V^2_\theta},\\
\|C^{\theta}_{\ge h}u\|_{L^q_tL^2_x} \lesssim h^{-\frac1q}\|u\|_{V^2_\theta}.
\end{aligned}
\end{align}
\end{lem}
We recall the uniform disposability of the modulation cutoff multipliers, which reads for $1\le q, p \le \infty$,
\begin{align}\label{uni-dis1}
\|C^\theta_{\le h}P_\lambda R_\kappa u\|_{L^q_tL^p_x} + \|C^\theta_h P_{\lambda}R_\kappa u\|_{L^q_tL^p_x} \lesssim \|P_\lambda R_\kappa u\|_{L^q_tL^p_x},
\end{align}
if $\kappa\in\mathcal C_{\mathtt r},\ h \gtrsim {\mathtt r}^2\lambda,$ and ${\mathtt r}\gtrsim\lambda^{-1}$.
Since convolution with $L^1_t(\mathbb R)$ functions is bounded on the $V^2$ space, we also have for every $h\in2^{\mathbb Z}$,
\begin{align}\label{uni-dis}
\|C^\theta_{\le h}u\|_{V^2_\theta} \lesssim \|u\|_{V^2_\theta}.
\end{align}
\begin{lem}[Lemma 7.3 of \cite{candyherr}]\label{lem-v-dual}
Let $F\in L^\infty_tL^2_x$, and suppose that
$$
\sup_{\|P_\lambda H_Nv\|_{V^2_\theta}\lesssim1}\left|\int_{\mathbb R} \langle P_\lambda H_Nv(t),F(t)\rangle_{L^2_x}\,dt \right| <\infty.
$$	
If $u\in C(\mathbb R,L^2_x)$ satisfies $-(i\partial_t+\theta\Lambda) u=F$, then $P_\lambda H_Nu\in V^2_\theta$ and we have the bound
\begin{align}\label{v-dual}
\|P_\lambda H_Nu\|_{V^2_\theta} \lesssim \|P_\lambda H_Nu(0)\|_{L^2_x} + \sup_{\|P_\lambda H_Nv\|_{V^2_\theta}\lesssim1}\left|\int_{\mathbb R} \langle P_\lambda H_Nv(t),F(t)\rangle_{L^2_x}\,dt \right|.
\end{align}
\end{lem}

\subsection{Auxiliary estimates}
We present several estimates which will play a key role in the proof of our main result. 
\begin{prop}[Lemma 3.5 of \cite{choozlee}]\label{stri-2d}
Let $(q,r)$ satisfy that $\frac1q+\frac1p=\frac12$. Then
$$
\|e^{\theta it\Lambda}P_\lambda f\|_{L^q_tL^p_x(\mathbb R^{1+2})} \lesssim \lambda^{\frac2q}\|P_\lambda f\|_{L^2_x(\mathbb R^2)}
$$	
for all $\lambda\ge1$.
\end{prop}

\begin{prop}[Proposition 3.1, 3.2 of \cite{cholee}]
Let $P_{\lambda_j} \Pi_{\theta_j}\psi =: \psi_{\lambda_j}\in V^2_{\theta_j}$ for $\theta_j \in \{+, -\}$ and $j = 1,2$. Then we have
\begin{align}
	\|P_{\lambda_0}(\psi_{\lambda_1}^\dagger\beta\psi_{\lambda_2})\|_{L^2_tL^2_x(\mathbb R^{1+2})} \lesssim \lambda_1^\eta \lambda_2^{1-\eta}\|\psi_{\lambda_1}\|_{V^2_{\theta_1}}\|\psi_{\lambda_2}\|_{V^2_{\theta_2}}, \label{bi-2d-lh}
\end{align}
for any $0<\eta<1$.
In particular, for High$\times$High interaction, i.e., $\lambda_0\lesssim\lambda_1\approx\lambda_2$, we have the following estimates:
\begin{align}
	\|P_{\lambda_0}(\psi_{\lambda_1}^\dagger\beta\psi_{\lambda_2})\|_{L^2_tL^2_x(\mathbb R^{1+2})} & \lesssim \lambda_0^\frac12\left(\frac{\lambda_0}{\lambda_1}\right)^\frac12 \|\psi_{\lambda_1}\|_{V^2_{\theta_1}}\|\psi_{\lambda_2}\|_{V^2_{\theta_2}},\quad \theta_1=\theta_2,\label{bi-2d-hh} \\
	\|P_{\lambda_0}(\psi_{\lambda_1}^\dagger\beta\psi_{\lambda_2})\|_{L^2_tL^2_x(\mathbb R^{1+2})} & \lesssim \lambda_0^\frac12 \|\psi_{\lambda_1}\|_{V^2_{\theta_1}}\|\psi_{\lambda_2}\|_{V^2_{\theta_2}},\quad \theta_1\neq\theta_2.
\end{align}	
\end{prop}

\begin{prop}[Lemma 3.1 of \cite{behe}]
Let $2<q\le\infty$. If $0<\mu\le\lambda$, and $\frac1q+\frac1p=\frac12$, then for every $\mathsf q\in\mathcal Q_\mu$ we have
$$
\|e^{-\theta it\Lambda}P_{\mathsf q}P_\lambda f\|_{L^q_tL^p_x(\mathbb R^{1+3})} \lesssim (\mu\lambda)^\frac1q \|P_{\mathsf q}P_\lambda f\|_{L^2_x(\mathbb R^3)}.
$$	
\end{prop}

\begin{lem}[Lemma 8.5 of \cite{candyherr}]\label{ang-con}
Let $2\le p<\infty$, and $0\le s<\frac2p$. If $\lambda,N\ge1$, ${\mathtt r}\gtrsim\lambda^{-1}$, and $\kappa\in\mathcal C_{\mathtt r}$, then we have
$$
\|R_\kappa P_\lambda H_N f\|_{L^p_x(\mathbb R^3)} \lesssim ({\mathtt r} N)^s \|P_\lambda H_N f\|_{L^p_x(\mathbb R^3)}.
$$	
\end{lem}

\section{Multilinear estimates I: Three spatial dimensions}\label{sec-multi-3d}
In this section we shall prove Proposition \ref{main-tri-est} in $\mathbb R^{1+3}$.

\subsection{Trilinear estimates for $s=0,\,\sigma>0$}\label{multi-3d}
In this section we prove \eqref{0-tri}. By trivial extension we may assume that $I = \mathbb R$. Let $\Theta = (\theta, \lambda, N)$ and $$\psi_{\Theta} = P_{\lambda}H_{N} \psi_{\theta}.$$
Let $\Theta_j = (\theta_j, \lambda_j, N_j)$ for $j \in \{1, 2, 3, 4\}$.
By definition of $F^{0, \sigma}$ and $V_\theta^2$ and Lemma \ref{lem-v-dual} we write
\begin{align*}
&\|\mathcal I_{\theta}[V_b*(\varphi^\dagger\beta\phi)\beta\psi](t_0; \cdot)\|_{F^{0, \sigma}}^2\\
&\qquad = \sum_{\lambda_4,N_4\in2^\mathbb N}N_4^{2\sigma}\sum_{\theta' \in \{+, -\}}\|\Pi_{\theta'}	P_{\lambda_4} H_{N_4}  \mathcal I_{\theta}[V_b*(\varphi^\dagger\beta\phi)\beta\psi](t_0; \cdot)\|_{V^2_{\theta'}}^2 \\
&\qquad = \sum_{\lambda_4,N_4\in2^\mathbb N}N_4^{2\sigma}\|\Pi_{\theta}	P_{\lambda_4} H_{N_4} \mathcal I_\theta[V_b*(\varphi^\dagger\beta\phi)\beta\psi](t_0; \cdot)\|_{V^2_\theta}^2 \\
&\qquad = \sum_{\lambda_4,N_4\ge1}N_4^{2\sigma}\sup_{\|\uppsi_{\Theta_4}\|_{V^2_\theta\le1}}\left|\int_{\mathbb R^{1+3}} [V_b*(\varphi^\dagger\beta\phi)] (\uppsi_{\Theta_4}^\dagger\beta\psi )\,dtdx \right|^2 \\
&\qquad \le \sum_{\lambda_4,N_4\ge1}N_4^{2\sigma}\left(\sum_{\Theta_j, j=1,2,3}\sup_{\|\uppsi_{\Theta_4}\|_{V^2_{\theta_4}}\lesssim1}\left|\int_{\mathbb R^{1+3}} \Lambda^{-2}(\varphi^\dagger_{\Theta_1}\beta\phi_{\Theta_2})(\uppsi_{\Theta_4}^\dagger \beta\psi_{\Theta_3})\,dtdx  \right| \right)^2 \\
&\qquad=: \mathfrak I_{\lambda,N}.
\end{align*}
We shall decompose $\mathfrak I_{\lambda,N}$ into the modulation $h$, which is the distance to the light cone in the space-time Fourier variables. Then we consider the high frequency and the low frequency cases, compared to the size of the modulation. That is, we are concerned with the following frequency interactions for $\{j,k\}=\{1,2\}$ or $\{3,4\}$,
$$
\min\{\lambda_0,\lambda_j,\lambda_k\} \lesssim \textrm{med}\{\lambda_0,\lambda_j,\lambda_k\} \approx\max\{\lambda_0,\lambda_j,\lambda_k\},
$$
together with
$$
h \gtrsim \lambda_{\max},\quad h\ll \lambda_{\max},
$$
where $\lambda_{\max}=\max\{\lambda_1,\lambda_2,\lambda_3,\lambda_4\}$. We shall write
\begin{align*}
\mathfrak I_{\lambda,N}& = \textrm{(Low frequency)}+\textrm{(High frequency)}  \\
&=: \mathfrak I_{h\gtrsim\lambda}+\mathfrak I_{h\ll\lambda}.
\end{align*}

To estimate the low frequency part $\mathfrak I_{h\gtrsim\lambda}$,
 we write
\begin{align*}
\mathcal I_{\lambda,N} &:= \left|\int_{\mathbb R^{1+3}}  \Lambda^{-2}\Big( \left[ C_{\ge h}^{\theta_1}\varphi_{\Theta_1}\right]^\dagger\beta\phi_{\Theta_2}\Big)(\uppsi_{\Theta_4}^\dagger\beta\psi_{\Theta_3})\,dtdx\right|.	
\end{align*}
Then for small $0<\eta\ll1$, we have
\begin{align*}
&\sum_{h \gtrsim\lambda_{\max}}\mathcal I_{\lambda,N} \\
&= \sum_{h \gtrsim\lambda_{\max}}(\mathcal I_{\lambda,N})^\eta	(\mathcal I_{\lambda,N})^{1-\eta} \\
& \lesssim \left(\frac{\lambda_0}{\lambda_{\max}}\right)^{\frac12-3\eta}\left(\lambda_1^{-\frac12}\lambda_2^{-\frac12}\lambda_3^{-\frac12}\|\varphi_{\Theta_1}\|_{L^4_{t,x}}\|\phi_{\Theta_2}\|_{L^4_{t,x}}\|\psi_{\Theta_3}\|_{L^4_{t,x}}\right)^\eta \\
&\qquad\qquad \times\|\varphi_{\Theta_1}\|_{V^2_{\theta_1}}\|\phi_{\Theta_2}\|_{V^2_{\theta_2}}\|\psi_{\Theta_3}\|_{V^2_{\theta_3}}\|\uppsi_{\Theta_4}\|_{V^2_{\theta_4}}.
\end{align*}
Here the estimate
 $$
 (\mathcal I_{\lambda,N})^\eta\lesssim \left(\frac{\lambda_0}{\lambda_{\max}} \right)^{-2\eta}\left(\lambda_1^{-\frac12}\lambda_2^{-\frac12}\lambda_3^{-\frac12}\lambda_4^{-\frac12}\|\varphi_{\Theta_1}\|_{L^4_{t,x}}\|\phi_{\Theta_2}\|_{L^4_{t,x}}\|\psi_{\Theta_3}\|_{L^4_{t,x}}\|\uppsi_{\lambda_4,N_4}\|_{L^4_{t,x}}\right)^\eta
 $$
 follows from H\"older's inequality and $L^4$-Strichartz estimates and the estimate of $(\mathcal I_{\lambda,N})^{1-\eta}$ follows from \cite{chohlee}. This completes the estimate of $\mathfrak I_{h\gtrsim\lambda}$.

 On the other hand, the estimate of $\mathfrak I_{h\ll\lambda}$ requires more work. In this case the required estimates \eqref{0-tri} follows from the following $L^2$-bilinear estimates for $\frac1{10}<\delta<\frac12$ and $\eta<\frac\sigma{10}$,
 \begin{align}\label{bi1}
\begin{aligned}
\|P_{\lambda_0}H_{N_0}(\varphi^\dagger_{\Theta_1}\beta\phi_{\Theta_2})\|_{L^2_tL^2_x} & \lesssim (\lambda_1^{-\frac12}\|\varphi_{\Theta_1}\|_{L^4_tL^4_x}\lambda_2^{-\frac12}\|\phi_{\Theta_2}\|_{L^4_tL^4_x}	)^{\eta} \\
& \quad\qquad\times \lambda_0\left(\frac{\lambda_{\min}}{\lambda_{\max}}\right)^\delta (N_{\min})^\sigma \left(\|\varphi_{\Theta_1}\|_{V^2_{\theta_1}}\|\phi_{\Theta_2}\|_{V^2_{\theta_2}}\right)^{1-\eta}.
\end{aligned}
\end{align}
Indeed, by \eqref{bi1} we have
\begin{align*}
&\|\mathcal I_{\theta}[V_b*(\varphi^\dagger\beta\phi)\beta\psi](t_0; \cdot)\|_{F^{0, \sigma}}^2\\
&\qquad \lesssim \sum_{\lambda_4,N_4}N_4^{2\sigma}\bigg( \sum_{\Theta_j, j=1,2,3}(\lambda_1^{-\frac12}\lambda_2^{-\frac12}\lambda_3^{-\frac12} \|\varphi_{\Theta_1}\|_{L^4_{t,x}}\|\phi_{\Theta_2}\|_{L^4_{t,x}}\|\psi_{\Theta_3}\|_{L^4_{t,x}})^\eta  \\
& \qquad\qquad 	\times \left(\frac{\lambda_{\min}\lambda_{\min}'}{\lambda_{\max}^{12}\lambda_{\max}^{34}}\right)^\delta (N_{\min}N_{\min}')^\sigma (\|\varphi_{\Theta_1}\|_{V^2_{\theta_1}}\|\phi_{\Theta_2}\|_{V^2_{\theta_2}}\|\psi_{\Theta_3}\|_{V^2_{\theta_3}})^{1-\eta} \bigg)^2 \\
&\qquad =: \mathfrak S(\varphi,\phi,\psi).
\end{align*}
Even if the square summation seems quite complicated, it would be straightforward with a simple observation: By H\"older inequality in $\ell^p$-space, we get
\begin{align*}
&\sum_{\lambda,N} \left[(\lambda^{-\frac12}N^\sigma\|\varphi_{\Theta}\|_{L^4_{t,x}})^\eta  (N^\sigma \|\varphi_{\Theta}\|_{V^2_{\theta}})^{1-\eta}\right]^2\\
 &\qquad\lesssim \left\|(\lambda^{-\frac12}N^\sigma\|\varphi_{\Theta}\|_{L^4_{t,x}})\right\|_{\ell^{2}_{\lambda, N}}^\eta \left\|(N^\sigma \|\varphi_{\Theta}\|_{V^2_{\theta}})\right\|_{\ell^{2}_{\lambda, N}}^{1-\eta}.
\end{align*}
Therefore, in order to prove our first main theorem, it suffices to prove the estimates \eqref{bi1}. At the very beginning, we further decompose the frequency-localised bilinear form $\varphi^\dagger_{\lambda_1,N_1}\beta\phi_{\lambda_2,N_2}$ into the modulation localised form:
$$
P_{\lambda_0}H_{N_0}(\varphi^\dagger_{\Theta_1}\beta\phi_{\Theta_2} ) = \sum_{h\in2^\mathbb Z}(\mathcal A_0(h) + \mathcal A_1(h) + \mathcal A_2(h)),
$$
where
\begin{align*}
\mathcal A_0(h) & = C_hP_{\lambda_0}H_{N_0}(C^{\theta_1}_{\le h}\varphi_{\Theta_1})^\dagger\beta C^{\theta_2}_{\le h}\phi_{\Theta_2},\\
\mathcal A_1(h) & = C_{\le h}P_{\lambda_0}H_{N_0}(C^{\theta_1}_{ h}\varphi_{\Theta_1})^\dagger\beta C^{\theta_2}_{\le h}\phi_{\Theta_2},\\
\mathcal A_2(h) & = C_{\le h}P_{\lambda_0}H_{N_0}(C^{\theta_1}_{\le h}\varphi_{\Theta_1})^\dagger\beta C^{\theta_2}_{ h}\phi_{\Theta_2} .	
\end{align*}

We first consider the frequency relation in a low modulation regime, i.e.,
$$
\lambda_0\lesssim\lambda_1\approx\lambda_2, \ h \lesssim \lambda_0,\ \textrm{and}\ \lambda_1\lesssim\lambda_0\approx\lambda_2,\  h\lesssim\lambda_1.
$$
We observe that the square-summation over caps and cubes gives that
\begin{align}\label{phi-1}
\begin{aligned}
&\left( \sum_{\mathsf q\in\mathcal Q_{\lambda_0}}\sum_{\kappa\in\mathcal C_{\mathtt r}}\|P_{\mathsf q}R_\kappa\varphi_{\Theta_1}\|_{L^4_tL^4_x}^2 \right)^\frac12 \\
	&\qquad\qquad\qquad \lesssim {\mathtt r}^{-2\eta}\left(\frac{\lambda_0}{\lambda_1}\right)^{-2\eta}(\lambda_0\lambda_1)^\frac14\left(\lambda_1^{-\frac12}\|\varphi_{\Theta_1}\|_{L^4_tL^4_x}\right)^\eta \|\varphi_{\Theta_1}\|_{V^2_{\theta_1}}^{1-\eta}.
	\end{aligned}
\end{align}
In particular, we have
\begin{align}\label{phi-2}
\begin{aligned}
	\left( \sum_{\mathsf q\in\mathcal Q_{\lambda_0}}\sum_{\kappa\in\mathcal C_{\mathtt r}}\|P_{\mathsf q}R_\kappa\varphi_{\Theta_1}\|_{L^4_tL^4_x}^2 \right)^\frac12 \lesssim {\mathtt r}^{-\frac\epsilon2}\left(\frac{\lambda_0}{\lambda_1}\right)^{-\frac\epsilon2}(\lambda_0\lambda_1)^\frac14 \|\varphi_{\Theta_1}\|_{V^2_{\theta_1}}.
\end{aligned}	
\end{align}
The estimates of bilinear form $\varphi_{\Theta_1}^\dagger\beta\,\phi_{\Theta_2}$ in the low modulation regime is very similar to \cite{chohlee}. The only difference from our previous paper is that we apply \eqref{phi-1} instead of \eqref{phi-2}. To avoid unnecessary duplication we briefly present the proof of the bilinear form in this case. (See also Theorem 5.1 of \cite{chohlee}.) First we put ${\mathtt r}=(\frac{h\lambda_0}{\lambda_1\lambda_2})^\frac12$ and ${\mathtt r}^*=(\frac{h}{\lambda_0})^\frac12$.
For $\lambda_0\lesssim\lambda_1\approx\lambda_2$ with $h \lesssim \lambda_0$, we have
\begin{align*}
\|\mathcal A_0(h)\|_{L^2_tL^2_x} & \lesssim {\mathtt r}^{\sigma-4\eta}(h\lambda_0)^\frac12\left(\frac{\lambda_0}{\lambda_1}\right)^{\frac12-4\eta}(N_{\min}^{12})^\sigma\left(\lambda_1^{-\frac12}\|\varphi_{\Theta_1}\|_{L^4_tL^4_x}\lambda_2^{-\frac12}\|\phi_{\Theta_2}\|_{L^4_tL^4_x}\right)^\eta \\
& \qquad\qquad\times \left(\|\varphi_{\Theta_1}\|_{V^2_{\theta_1}}\|\phi_{\Theta_2}\|_{V^2_{\theta_2}}\right)^{1-\eta}.	
\end{align*}
 We also have
 \begin{align*}
 \|\mathcal A_0(h)\|_{L^2_tL^2_x} & = \sup_{\|\psi\|_{L^2_tL^2_x}\le 1}\left|\int C_h\psi_{\lambda_0,N_0} (C^{\theta_1}_{\le h}\varphi_{\Theta1})^\dagger\beta(C^{\theta_2}_{\le h}\phi_{\Theta_2})\,dtdx\right| \\
 & \lesssim 	\sup_{\|\psi\|_{L^2_tL^2_x}\le 1}\sum_{\substack{\kappa,\,\kappa'\in\mathcal C_{\mathtt r} \\ |\theta_1\kappa-\theta_2\kappa'|\lesssim{\mathtt r}}}\sum_{\substack{ \kappa''\in\mathcal C_{{\mathtt r}^*} \\ |\theta_1\kappa+\kappa''|\lesssim{\mathtt r}^* }}\sum_{\substack{\mathsf q,\, \mathsf q'\in\mathcal Q_{\lambda_0} \\ |\theta_1\mathsf q-
 \theta_2\mathsf q'|\lesssim\lambda_0}} \\
 & \qquad\qquad\left|\int_{\mathbb R^{1+3}} R_{\kappa''}C_h \psi_{\lambda_0,N_0} (R_\kappa P_{\mathsf q}C^{\theta_1}_{\le h}\varphi_{\Theta_1})^\dagger\beta(R_{\kappa'}P_{\mathsf q'}C^{\theta_2}_{\le h}\phi_{\Theta_2})\,dtdx\right| \\
 & \lesssim {\mathtt r}^{1-2\eta}({\mathtt r}^* N_0)^\sigma (\lambda_0\lambda_1)^\frac12\left(\frac{\lambda_0}{\lambda_1}\right)^{-4\eta} \left(\lambda_1^{-\frac12}\|\varphi_{\Theta_1}\|_{L^4_tL^4_x}\lambda_2^{-\frac12}\|\phi_{\Theta_2}\|_{L^4_tL^4_x}\right)^\eta \\
 & \qquad\qquad\times \left(\|\varphi_{\Theta_1}\|_{V^2_{\theta_1}}\|\phi_{\Theta_2}\|_{V^2_{\theta_2}}\right)^{1-\eta} \\
 & \lesssim (\mathtt r^*)^{\sigma-4\eta}(h\lambda_0)^\frac12 \left(\frac{\lambda_0}{\lambda_1}\right)^{\frac12-6\eta}(N_0)^\sigma\left(\lambda_1^{-\frac12}\|\varphi_{\Theta_1}\|_{L^4_tL^4_x}\lambda_2^{-\frac12}\|\phi_{\Theta_2}\|_{L^4_tL^4_x}\right)^\eta \\
 & \qquad\qquad\times \left(\|\varphi_{\Theta_1}\|_{V^2_{\theta_1}}\|\phi_{\Theta_2}\|_{V^2_{\theta_2}}\right)^{1-\eta}. 
  \end{align*}
In this section $\kappa, \kappa', q, q'$ in $|\theta \kappa + \theta'\kappa'|$ and $|\theta q + \theta' q'|$ denote the center points of the corresponding caps $\kappa, \kappa'$ and cubes $\mathsf q, \mathsf q'$. The summation on $h\lesssim\lambda_0$ gives the required bound.
 We can estimate $\mathcal A_1(h), \mathcal A_2(h)$ in a similar way. We omit the details.
  For $\lambda_1\lesssim\lambda_0\approx\lambda_2$, with $h\lesssim\lambda_1$, we have
\begin{align*}
\|\mathcal A_0(h)\|_{L^2_tL^2_x} & = 	\sup_{\|\psi\|_{L^2_tL^2_x}\lesssim1}\left|\int C_h\psi_{\lambda_0,N_0} (C^{\theta_1}_{\le h}\varphi_{\Theta_1})^\dagger\beta(C^{\theta_2}_{\le h}\phi_{\Theta_2})\,dtdx\right| \\
& \lesssim \sup_{\|\psi\|_{L^2_tL^2_x\lesssim1}}\sum_{\substack{\kappa,\,\kappa',\,\kappa''\in\mathcal C_{{\mathtt r}_{*}} \\ |\theta_1\kappa+\theta_2\kappa'|,|\kappa''+\theta_2\kappa'|\lesssim{\mathtt r}_{*}}}\sum_{\substack{\mathsf q',\, \mathsf q''\in\mathcal Q_{\lambda_1} \\ |\mathsf q' + \mathsf q''|\lesssim\lambda_1}}\\
& \qquad\qquad\left|\int R_{\kappa''}P_{\mathsf q''}C_h\psi_{\lambda_0,N_0} (R_\kappa C^{\theta_1}_{\le h}\varphi_{\Theta_1})^\dagger\beta(R_{\kappa'}P_{\mathsf q'}C^{\theta_2}_{\le h}\phi_{\Theta_2})\,dtdx\right| \\
& \lesssim \sup_{\|\psi\|_{L^2_tL^2_x\lesssim1}}\sum_{\substack{\kappa,\,\kappa',\,\kappa''\in\mathcal C_{{\mathtt r}_*} \\ |\theta_1\kappa+\theta_2\kappa'|,|\kappa''+\theta_2\kappa'|\lesssim {\mathtt r}_*}}\sum_{\substack{\mathsf q', \,\mathsf q''\in\mathcal Q_{\lambda_1} \\ |\mathsf q' + \mathsf q''|\lesssim\lambda_1}} \\
&\qquad\qquad {\mathtt r}_*\|R_{\kappa''}P_{\mathsf q''}C_h\psi_{\lambda_0,N_0}\|_{L^2_tL^2_x}\|R_\kappa C^{\theta_1}_{\le h}\varphi_{\Theta_1}\|_{L^4_tL^4_x}\|R_{\kappa'}P_{\mathsf q'}C_{\le h}^{\theta_2}\phi_{\Theta_2}\|_{L^4_tL^4_x} \\
& \lesssim ({\mathtt r}_*)^{1-2\eta}\left(\frac{\lambda_1}{\lambda_2} \right)^{-2\eta}\lambda_1^\frac12(\lambda_1\lambda_2)^\frac14({\mathtt r}_* N_{\min}^{012})^\sigma\left( \lambda_1^{-\frac12}\|\varphi_{\Theta_1}\|_{L^4_tL^4_x}\lambda_2^{-\frac12}\|\phi_{\Theta_2}\|_{L^4_tL^4_x} \right)^\eta \\
&\qquad\qquad\qquad\qquad\times \left( \|\varphi_{\Theta_1}\|_{V^2_{\theta_1}}\|\phi_{\Theta_2}\|_{V^2_{\theta_2}} \right)^{1-\eta} \\
& \lesssim (\mathtt r_*)^{\sigma-2\eta}h_1^\frac12\lambda_1^\frac14\lambda_2^\frac14 \left(\frac{\lambda_1}{\lambda_2}\right)^{-2\eta}(N_{\min}^{012})^\sigma\left( \lambda_1^{-\frac12}\|\varphi_{\Theta_1}\|_{L^4_tL^4_x}\lambda_2^{-\frac12}\|\phi_{\Theta_2}\|_{L^4_tL^4_x} \right)^\eta \\
&\qquad\qquad\qquad\qquad\times \left( \|\varphi_{\Theta_1}\|_{V^2_{\theta_1}}\|\phi_{\Theta_2}\|_{V^2_{\theta_2}} \right)^{1-\eta} ,
\end{align*}
where we chose ${\mathtt r}_*=(\frac{h}{\lambda_1})^\frac12$. The summation with respect to the modulation $h\lesssim\lambda_1$ yields the required estimates. As the case $\lambda_0\lesssim\lambda_1\approx\lambda_2$, we can treat the terms $\mathcal A_1(h)$ and $\mathcal A_2(h)$ in a similar manner. We omit the details.
\subsection{High modulation regime}
In this subsection we shall estimate the bilinear form $\varphi_{\Theta_1}^\dagger\beta\phi_{\Theta_2}$ in the regime: $\lambda_{\min}\ll h\ll\lambda_{\max}$.
We first consider $\lambda_0\lesssim\lambda_1\approx\lambda_2$ with $\lambda_0\ll h \ll\lambda_1$.
We observe that the angle between the diameter of support of $\widehat{\varphi}$ and $\widehat \phi$ is less than $\left(\frac{\lambda_0}{\lambda_1}\right)^\frac12$. We then decompose $\mathcal A_0(h)$ into the following:
\begin{align}
\begin{aligned}
\|\mathcal A_0(h)\|_{L^2_{t,x}} & \lesssim \|C_h P_{\lambda_0}H_{N_0}(C^{\theta_1}_{\approx h}\varphi_{\Theta_1})^\dagger\beta(C^{\theta_2}_{\ll h}\phi_{\Theta_2})\|_{L^2_{t,x}}  \\
&\qquad\qquad + \|C_h P_{\lambda_0}H_{N_0}(C^{\theta_1}_{\ll h}\varphi_{\Theta_1})^\dagger\beta(C^{\theta_2}_{\approx h}\phi_{\Theta_2})\|_{L^2_{t,x}} \\
&=: A_{0,1}+A_{0,2}. 	
\end{aligned}
\end{align}
It is enough to estimate the $A_{0,1}$ term because of symmetry. By almost orthogonal decomposition of small cubes of size $\lambda_0$ we see that
\begin{align*}
 A_{0,1} 
 & \lesssim \left(\frac{\lambda_0}{\lambda_1}\right)^\frac12\sum_{\substack{\mathsf q, \mathsf q'\in\mathcal Q_{\lambda_0} \\ |\mathsf q - \mathsf q'|\lesssim\lambda_0}}\|P_{\mathsf q}C^{\theta_1}_{\approx h}\varphi_{\Theta_1}\|_{L^4_{t,x}}\|P_{\mathsf q'}C^{\theta_2}_{\ll h}\phi_{\Theta_2}\|_{L^4_{t,x}} \\
 & \lesssim  \left(\frac{\lambda_0}{\lambda_1}\right)^\frac12 \bigg(\sum_{\mathsf q'}\|P_{\mathsf q'}C^{\theta_2}_{\ll h}\phi_{\Theta_2}\|_{L^4_{t,x}}^2 \bigg)^\frac12 \bigg( \sum_{\mathsf q'} \bigg(\sum_{\mathsf q:|\mathsf q - \mathsf q'|\lesssim\lambda_0}\|P_{\mathsf q}C^{\theta_1}_{\approx h}\varphi_{\Theta_1}\|_{L^4_{t,x}}\bigg)^2 \bigg)^\frac12 \\
 & \lesssim \left(\frac{\lambda_0}{\lambda_1}\right)^\frac12 \bigg(\sum_{\mathsf q'}\|P_{\mathsf q'}C^{\theta_2}_{\ll h}\phi_{\Theta_2}\|_{L^4_{t,x}}^2 \bigg)^\frac12  \bigg(\sum_{\mathsf q}\|P_{\mathsf q}C^{\theta_1}_{\approx h}\varphi_{\Theta_1}\|_{L^4_{t,x}}^2 \bigg)^\frac12.
\end{align*}
For the square-summation of $\phi_{\Theta_2}$ in the third inequality we apply similar form of \eqref{phi-1} to obtain
\begin{align}\label{square-sum-q}
	\bigg(\sum_{\mathsf q'}\|P_{{\mathsf q}'}C^{\theta_2}_{\ll h}\phi_{\Theta_2}\|_{L^4_{t,x}}^2 \bigg)^\frac12 & \lesssim \left(\frac{\lambda_0}{\lambda_2}\right)^{-2\eta}(\lambda_0\lambda_2)^\frac14 \left(\lambda_2^{-\frac12} \|\phi_{\Theta_2}\|_{L^4_{t,x}}\right)^\eta \|\phi_{\Theta_2}\|_{V^2_{\theta_2}}^{1-\eta}
\end{align}
and for the summation of $\varphi_{\Theta_1}$ we write
\begin{align}
	\bigg(\sum_{{\mathsf q}}&\|P_{{\mathsf q}}C^{\theta_1}_{\approx h}\varphi_{\Theta_1}\|_{L^4_{t,x}}^2 \bigg)^\frac12\nonumber\\
 & \lesssim \left( \sum_{{\mathsf q}}\left(\|P_{{\mathsf q}}C^{\theta_1}_{\approx h}\varphi_{\Theta_1}\|_{L^4_{t,x}}^\eta \|P_{{\mathsf q}}C^{\theta_1}_{\approx h}\varphi_{\Theta_1}\|_{L^4_{t,x}}^{1-\eta} \right)^2 \right)^\frac12 \nonumber\\
	& \lesssim \left(\sum_{\mathsf q}\|P_{{\mathsf q}}C^{\theta_1}_{\approx h}\varphi_{\Theta_1}\|_{L^4_{t,x}}^2\right)^\frac\eta2 \left(\sum_{\mathsf q} \|P_{{\mathsf q}}C^{\theta_1}_{\approx h}\varphi_{\Theta_1}\|_{L^4_{t,x}}^2 \right)^{\frac{1-\eta}{2}}\nonumber \\
	& \lesssim \left(\sum_{\mathsf q}\|P_{{\mathsf q}}C^{\theta_1}_{\approx h}\varphi_{\Theta_1}\|_{L^4_{t,x}}^2\right)^\frac\eta2 \lambda_0^{\frac34(1-\eta)} \left(\sum_{\mathsf q} \|P_{{\mathsf q}}C^{\theta_1}_{\approx h}\varphi_{\Theta_1}\|_{L^4_{t}L^2_x}^2 \right)^{\frac{1-\eta}{2}} \nonumber\\
	& \lesssim \left(\sum_{\mathsf q}\|P_{{\mathsf q}}C^{\theta_1}_{\approx h}\varphi_{\Theta_1}\|_{L^4_{t,x}}^2\right)^\frac\eta2 \lambda_0^{\frac34(1-\eta)}h^{-\frac14(1-\eta)}\|\varphi_{\Theta_1}\|_{V^2_{\theta_1}}^{1-\eta} \nonumber\\
	& \lesssim \lambda_1^\frac\eta2\left(\frac{\lambda_0}{\lambda_1}\right)^{-2\eta}\left(\frac{\lambda_0}{h}\right)^{\frac14(1-\eta)}\lambda_0^{\frac12(1-\eta)}\left(\lambda_1^{-\frac12}\|\varphi_{\Theta_1}\|_{L^4_tL^4_x}\right)^\eta \|\varphi_{\Theta_1}\|_{V^2_{\theta_1}}^{1-\eta}.\label{high-mod-eta}
\end{align}
Here we used the Bernstein inequality, and the bound of high modulation regime \eqref{bdd-high-mod} and then \eqref{square-sum-q}.
Thus we get
\begin{align*}
\sum_{\lambda_0\ll h\ll\lambda_1}A_{0,1} & \lesssim \lambda_0\left( \frac{\lambda_0}{\lambda_1}\right)^{\frac14-5\eta}	\left( \lambda_1^{-\frac12}\|\varphi_{\Theta_1}\|_{L^4_tL^4_x}\lambda_2^{-\frac12}\|\phi_{\Theta_2}\|_{L^4_tL^4_x} \right)^\eta \\
&\qquad\qquad\qquad\qquad\qquad\times \left( \|\varphi_{\Theta_1}\|_{V^2_{\theta_1}}\|\phi_{\Theta_2}\|_{V^2_{\theta_2}} \right)^{1-\eta}.
\end{align*}
For $\mathcal A_1(h)$, we have the similar decomposition as follows
\begin{align*}
\|\mathcal A_1(h)\|_{L^2_tL^2_x} & \lesssim \|C_{\approx h}P_{\lambda_0}H_{N_0}(C^{\theta_1}_{ h}\varphi_{\Theta_1})^\dagger\beta C^{\theta_2}_{\ll h}\phi_{\Theta_2}\|_{L^2_tL^2_x} \\
& \qquad\qquad + \|C_{\ll h}P_{\lambda_0}H_{N_0}(C^{\theta_1}_{ h}\varphi_{\Theta_1})^\dagger\beta C^{\theta_2}_{\approx h}\phi_{\Theta_2}\|_{L^2_tL^2_x} \\
&=: A_{1,1}+A_{1,2}.
\end{align*}
The term $A_{1,1}$ can be treated in the identical manner as $A_{0,1}$. For $A_{1,2}$, we write
\begin{align*}
A_{1,2} & \lesssim \left(\frac{\lambda_0}{\lambda_1}\right)^\frac12\sum_{\substack{{\mathsf q},\,{\mathsf q}'\in\mathcal Q_{\lambda_0} \\ |{\mathsf q}-{\mathsf q}'|\lesssim\lambda_0}}	\|C_{\ll h}P_{\lambda_0}H_{N_0}(P_qC^{\theta_1}_{ h}\varphi_{\Theta_1})^\dagger\beta P_{{\mathsf q}'}C^{\theta_2}_{\approx h}\phi_{\Theta_2}\|_{L^2_tL^2_x} \\
& \lesssim \left(\frac{\lambda_0}{\lambda_1}\right)^\frac12 \left(\sum_{\mathsf q}\|P_qC^{\theta_1}_h\varphi_{\Theta_1}\|_{L^4_tL^4_x}^2 \right)^\frac12 \left( \sum_{{\mathsf q}'}\|P_{{\mathsf q}'}C^{\theta_2}_{\approx h}\phi_{\Theta_2}\|_{L^4_tL^4_x}^2 \right)^\frac12.
\end{align*}
We use \eqref{high-mod-eta} to obtain
\begin{align*}
A_{1,2} & \lesssim \left( \frac{\lambda_0}{\lambda_1} \right)^\frac12 \lambda_1^\eta \left(\frac{\lambda_0}{\lambda_1}\right)^{-4\eta}\left( \frac{\lambda_0}{h}\right)^{\frac12(1-\eta)}\lambda_0^{1-\eta}	\left( \lambda_1^{-\frac12}\|\varphi_{\Theta_1}\|_{L^4_tL^4_x}\lambda_2^{-\frac12}\|\phi_{\Theta_2}\|_{L^4_tL^4_x} \right)^\eta \\
&\qquad\qquad\qquad\qquad\times \left( \|\varphi_{\Theta_1}\|_{V^2_{\theta_1}}\|\phi_{\Theta_2}\|_{V^2_{\theta_2}} \right)^{1-\eta} \\
& \lesssim \lambda_0\left( \frac{\lambda_0}{h}\right)^{\frac12(1-\eta)}\left(\frac{\lambda_0}{\lambda_1}\right)^{\frac12-5\eta}\left( \lambda_1^{-\frac12}\|\varphi_{\Theta_1}\|_{L^4_tL^4_x}\lambda_2^{-\frac12}\|\phi_{\Theta_2}\|_{L^4_tL^4_x} \right)^\eta \\
&\qquad\qquad\qquad\qquad\times \left( \|\varphi_{\Theta_1}\|_{V^2_{\theta_1}}\|\phi_{\Theta_2}\|_{V^2_{\theta_2}} \right)^{1-\eta},
\end{align*}
and hence the summation with respect to the modulation $\lambda_0\ll h\ll\lambda_1$ gives the desired estimate.
Now we consider the case $\lambda_1\lesssim\lambda_0\approx\lambda_2$ with $\lambda_1\ll h\ll \lambda_0$.
\begin{align*}
A_{0,1} & \lesssim \sup_{\|\psi\|_{L^2_tL^2_x}\lesssim1}	\sum_{\substack{{\mathsf q}',\,{\mathsf q}''\in\mathcal Q_{\lambda_1} \\ |{\mathsf q}'-{\mathsf q}''|\lesssim\lambda_1}}\|P_{{\mathsf q}''}C_h\psi_{\lambda_0,N_0}\|_{L^2_tL^2_x}\|C^{\theta_1}_{\approx h}\varphi_{\Theta_1}\|_{L^4_tL^4_x}\|P_{{\mathsf q}'}C^{\theta_2}_{\ll h}\phi_{\Theta_2}\|_{L^4_tL^4_x} \\
& \lesssim \|C^{\theta_1}_{\approx h}\varphi_{\Theta_1}\|_{L^4_tL^4_x}\left( \sum_{{\mathsf q}'}\|P_{{\mathsf q}'}C^{\theta_2}_{\ll h}\phi_{\Theta_2}\|_{L^4_tL^4_x}^2 \right)^\frac12.
\end{align*}
Here, for the $\varphi_{\Theta_1}$, we simply use the Bernstein's inequality, boundedness \eqref{bdd-high-mod}, and the uniform disposability \eqref{uni-dis1}   to get
\begin{align}\label{stri-mod-eta}
\begin{aligned}
\|C^{\theta_1}_{\approx h}\varphi_{\Theta_1}\|_{L^4_tL^4_x} & = \|C^{\theta_1}_{\approx h}\varphi_{\Theta_1}\|_{L^4_tL^4_x}^\eta\|C^{\theta_1}_{\approx h}\varphi_{\Theta_1}\|_{L^4_tL^4_x}^{1-\eta} \\
& \lesssim \|C^{\theta_1}_{\approx h}\varphi_{\Theta_1}\|_{L^4_tL^4_x}^\eta \lambda_1^{\frac34(1-\eta)} \|C^{\theta_1}_{\approx h}\varphi_{\Theta_1}\|_{L^4_tL^2_x}^{1-\eta} \\
& \lesssim \lambda_1^{\frac\eta2}\left(\lambda_1^{-\frac12}\|\varphi_{\Theta_1}\|_{L^4_tL^4_x}\right)^\eta \lambda_1^{\frac34(1-\eta)}h^{-\frac14(1-\eta)}\|\varphi_{\Theta_1}\|_{V^2_{\theta_1}}^{1-\eta}	,
\end{aligned}
\end{align}
 We recall \eqref{square-sum-q} and get
\begin{align*}
A_{0,1}	& \lesssim \left(\frac{\lambda_1}{h}\right)^{\frac14(1-\eta)}\lambda_1^\frac34\lambda_0^\frac14\left(\frac{\lambda_1}{\lambda_0}\right)^{-2\eta}\left( \lambda_1^{-\frac12}\|\varphi_{\Theta_1}\|_{L^4_tL^4_x}\lambda_2^{-\frac12}\|\phi_{\Theta_2}\|_{L^4_tL^4_x} \right)^\eta \\
&\qquad\qquad\qquad\qquad\times \left( \|\varphi_{\Theta_1}\|_{V^2_{\theta_1}}\|\phi_{\Theta_2}\|_{V^2_{\theta_2}} \right)^{1-\eta} \\
& \lesssim \lambda_0 \left( \frac{\lambda_1}{h} \right)^{\frac14(1-\eta)}\left(\frac{\lambda_1}{\lambda_0}\right)^{\frac34-2\eta}\left( \lambda_1^{-\frac12}\|\varphi_{\Theta_1}\|_{L^4_tL^4_x}\lambda_2^{-\frac12}\|\phi_{\Theta_2}\|_{L^4_tL^4_x} \right)^\eta \\
&\qquad\qquad\qquad\qquad\times \left( \|\varphi_{\Theta_1}\|_{V^2_{\theta_1}}\|\phi_{\Theta_2}\|_{V^2_{\theta_2}} \right)^{1-\eta}.
\end{align*}
Similarly, for $A_{0,2}$ we write
\begin{align*}
A_{0,2} & \lesssim \|C^{\theta_1}_{\ll h}\varphi_{\Theta_1}\|_{L^4_tL^4_x}\left(\sum_{{\mathsf q}'\in\mathcal Q_{\lambda_1}}\|P_{{\mathsf q}'}C^{\theta_2}_{\approx h}\phi_{\Theta_2}\|_{L^4_tL^4_x}^2\right)^\frac12.	
\end{align*}
Then we follow a similar argument to the estimate of \eqref{high-mod-eta} to get
\begin{align*}
	\left(\sum_{{\mathsf q}'\in\mathcal Q_{\lambda_1}}\|P_{{\mathsf q}'}C^{\theta_2}_{\approx h}\phi_{\Theta_2}\|_{L^4_tL^4_x}^2\right)^\frac12 & \lesssim \lambda_2^\frac\eta2\left(\frac{\lambda_1}{\lambda_0}\right)^{-2\eta}\left(\frac{\lambda_1}{h}\right)^{\frac14(1-\eta)}\lambda_1^{\frac12(1-\eta)}\left( \lambda_2^{-\frac12}\|\phi_{\Theta_2}\|_{L^4_tL^4_x} \right)^\eta \|\phi_{\Theta_2}\|_{V^2_{\theta_2}}^{1-\eta}.
\end{align*}
On the other hand, for $\varphi_{\Theta_1}$, we  use the uniform disposability \eqref{uni-dis1} and $L^4$-Strichartz estimates to obtain
\begin{align*}
\|C^{\theta_1}_{\ll h}\varphi_{\Theta_1}\|_{L^4_tL^4_x} & \lesssim 	\|\varphi_{\Theta_1}\|_{L^4_tL^4_x}^{\eta}\|\varphi_{\Theta_1}\|_{L^4_tL^4_x}^{1-\eta} \\
& \lesssim \lambda_1^\frac12 (\lambda_1^{-\frac12}\|\varphi_{\Theta_1}\|_{L^4_tL^4_x})^\eta \|\varphi_{\Theta_1}\|_{V^2_{\theta_1}}^{1-\eta}.
\end{align*}
Combining these, we get
\begin{align*}
A_{0,2} & \lesssim \lambda_0\left(\frac{\lambda_1}{\lambda_0}\right)^{1-2\eta}\left(\frac{\lambda_1}{h}\right)^{\frac14(1-\eta)}\left( \lambda_1^{-\frac12}\|\varphi_{\Theta_1}\|_{L^4_tL^4_x}\lambda_2^{-\frac12}\|\phi_{\Theta_2}\|_{L^4_tL^4_x} \right)^\eta \\
&\qquad\qquad\qquad\qquad\times \left( \|\varphi_{\Theta_1}\|_{V^2_{\theta_1}}\|\phi_{\Theta_2}\|_{V^2_{\theta_2}} \right)^{1-\eta}.	
\end{align*}
Thus the summation of $A_{0,1}$ and $A_{0,2}$ with respect to $\lambda_1\ll h\ll\lambda_0$ gives the required bound.
We now consider $\mathcal A_1(h)$. We decompose it into
\begin{align*}
\|\mathcal A_1(h)\|_{L^2_tL^2_x} & \lesssim \|C_{\approx h}P_{\lambda_0}H_{N_0}(C^{\theta_1}_{ h}\varphi_{\Theta_1})^\dagger\beta C^{\theta_2}_{\ll h}\phi_{\Theta_2}\|_{L^2_tL^2_x} \\
& \qquad\qquad + \|C_{\ll h}P_{\lambda_0}H_{N_0}(C^{\theta_1}_{ h}\varphi_{\Theta_1})^\dagger\beta C^{\theta_2}_{\approx h}\phi_{\Theta_2}\|_{L^2_tL^2_x} \\
&:= A_{1,1}+A_{1,2}.
\end{align*}
Since $A_{1,1}$ can be treated in the identical manner as $A_{0,1}$, we only treat the $A_{1,2}$ term. By orthogonal decompositions by smaller cubes of size $\lambda_1$ and following the argument in \eqref{stri-mod-eta} and estimate \eqref{high-mod-eta} we have
\begin{align*}
A_{1,2} & \lesssim \|C^{\theta_1}_h \varphi_{\Theta_1}\|_{L^4_tL^4_x}\left(\sum_{{\mathsf q}'\in\mathcal Q_{\lambda_1}}\|P_{{\mathsf q}'}C^{\theta_2}_{\approx h}\phi_{\Theta_2}\|_{L^4_tL^4_x}^2 \right)^\frac12	\\
& \lesssim  \lambda_1^{\frac\eta2}(\lambda_1^{-\frac12}\|\varphi_{\Theta_1}\|_{L^4_tL^4_x})^\eta \lambda_1^{\frac34(1-\eta)}h^{-\frac14(1-\eta)}\|\varphi_{\Theta_1}\|_{V^2_{\theta_2}}^{1-\eta}	 \\
&\times \lambda_2^\frac\eta2\left(\frac{\lambda_1}{\lambda_0}\right)^{-\epsilon}\left(\frac{\lambda_1}{h}\right)^{\frac14(1-\eta)}\lambda_1^{\frac12(1-\eta)}\left( \lambda_2^{-\frac12}\|\phi_{\Theta_2}\|_{L^4_tL^4_x} \right)^\eta \|\phi_{\Theta_2}\|_{V^2_{\theta_2}}^{1-\eta} \\
& \lesssim \lambda_1 \left(\frac{\lambda_1}{h} \right)^{\frac12(1-\eta)}\left( \frac{\lambda_1}{\lambda_2} \right)^{-\eta}\left(\lambda_1^{-\frac12}\lambda_2^{-\frac12}\|\varphi_{\Theta_1}\|_{L^4_tL^4_x}\|\phi_{\Theta_2}\|_{L^4_tL^4_x} \right)^\eta \\
& \qquad\qquad \times \left(\|\varphi_{\Theta_1}\|_{V^2_{\theta_1}}\|\phi_{\Theta_2}\|_{V^2_{\theta_2}} \right)^{1-\eta}.
\end{align*}
Thus we have 
$$
\sum_{\lambda_1\ll h\ll\lambda_0}A_{1,2} \lesssim \lambda_1\left(\frac{\lambda_1}{\lambda_0}\right)^{\frac12(1-3\eta)}\left(\lambda_1^{-\frac12}\lambda_2^{-\frac12}\|\varphi_{\Theta_1}\|_{L^4_tL^4_x}\|\phi_{\Theta_2}\|_{L^4_tL^4_x} \right)^\eta  \left(\|\varphi_{\Theta_1}\|_{V^2_{\theta_1}}\|\phi_{\Theta_2}\|_{V^2_{\theta_2}} \right)^{1-\eta}.
$$
The estimate of $\mathcal A_2(h)$ is then given straightforwardly. Indeed, we decompose $\mathcal A_2(h)$ into
\begin{align*}
\|\mathcal A_2(h)\|_{L^2_tL^2_x} & \lesssim \|C_{\approx h}P_{\lambda_0}H_{N_0}(C^{\theta_1}_{\ll h}\varphi_{\Theta_1})^\dagger\beta C^{\theta_2}_{ h}\phi_{\Theta_2}\|_{L^2_tL^2_x} \\
& \qquad\qquad + \|C_{\ll h}P_{\lambda_0}H_{N_0}(C^{\theta_1}_{\approx  h}\varphi_{\Theta_1})^\dagger\beta C^{\theta_2}_{ h}\phi_{\Theta_2}\|_{L^2_tL^2_x} \\
&:= A_{2,1}+A_{2,2}.
\end{align*}
We see that $A_{2,1}$ is essentially the same as $A_{0,2}$ whereas the $A_{2,2}$ is same as $A_{1,2}$. Hence we complete the proof of \eqref{0-tri} when $d=3$.


\section{Multilinear estimates I\!I: Two spatial dimensions}\label{multi-2d}
Since the angular regularity is not involved in our analysis, we use the notation $\Theta = (\theta, \lambda)$ and $\psi_\Theta = P_{\lambda}\psi_\theta$ for $\lambda \in 2^{\mathbb N}, \theta \in \{+, -\}$.
We write
\begin{align*}
\mathfrak I_\lambda&:=\sum_{\lambda_4\ge1}\left(  \sup_{\|\uppsi\,\,\|_{V^2_{\theta_4}}\le 1}\sum_{\lambda_0, \Theta_j}\bigg|\int_{\mathbb R^{1+2}}\Lambda^{-2}P_{\lambda_0}(\varphi_{\Theta_1}^\dagger\beta\phi_{\Theta_2})\widetilde P_{\lambda_0}(\uppsi_{\Theta_4}^\dagger\beta\psi_{\Theta_3})\,dtdx \bigg| \right)^2 \\
& \le (\textrm{Low modulation})+(\textrm{High modulation})	\\
& =: \mathfrak I_{\lambda\gg h}+\mathfrak I_{\lambda\lesssim h}.
\end{align*}
The estimate of $\mathfrak I_{\lambda\lesssim h}$ is essentially the same as the three dimensional case. We omit the details.
In advance we note that in the low-modulation-regime, i.e., $\lambda\gg d$ with High$\times$High modulation such as $\lambda_0\lesssim\lambda_1\approx\lambda_2$ and $\lambda_0\lesssim\lambda_3\approx\lambda_4$, we must have $\theta_1=\theta_2$ and $\theta_3=\theta_4$. We also observe that the estimate of $\mathfrak I_{\lambda\gg h}$ is very straightforward in view of the previous result \cite{cholee}. It only needs a small modification in multilinear estimates. Indeed, we write
\begin{align*}
\mathfrak I_{\lambda\gg h} & \lesssim \sum_{\lambda_4\ge1}\left(\sup_{\|\uppsi\,\,\|_{V^2_{\theta_4}}=1}\left(\mathbf I_1+\mathbf I_2+\mathbf I_3\right)\right)^2,	
\end{align*}
where
\begin{align*}
\mathbf I_1 = \sum_{\substack{\lambda_1,\lambda_2\ge1 \\ \lambda_0\ll\lambda_3\approx\lambda_4}}|\cdots|,\quad \mathbf I_2=\sum_{\substack{\lambda_1,\lambda_2\ge1 \\  \lambda_3\ll\lambda_0\approx\lambda_4}}|\cdots|,\quad \mathbf I_3=\sum_{\substack{\lambda_1,\lambda_2\ge1 \\ \lambda_4\ll\lambda_0\approx\lambda_3}}|\cdots|.
\end{align*}
We further divide $\mathbf I_j$, $j=1,2,3$ as follows:
\begin{align*}
\mathbf I_1 & \lesssim \sum_{\substack{\lambda_0\ll\lambda_1\approx\lambda_2 \\ \lambda_0\ll\lambda_3\approx\lambda_4}}|\cdots|+\sum_{\substack{\lambda_1\ll\lambda_0\approx\lambda_2 \\ \lambda_0\ll\lambda_3\approx\lambda_4}}|\cdots| + \sum_{\substack{\lambda_2\ll\lambda_0\approx\lambda_1\\ \lambda_0\ll\lambda_3\approx\lambda_4}}|\cdots| \\
& =: \mathbf I_{11}+\mathbf I_{12}+\mathbf I_{13},\\
\mathbf I_2 &\lesssim \sum_{\substack{\lambda_0\ll\lambda_1\approx\lambda_2 \\ \lambda_3\ll\lambda_0\approx\lambda_4}}|\cdots|+\sum_{\substack{\lambda_1\ll\lambda_0\approx\lambda_2 \\ \lambda_3\ll\lambda_0\approx\lambda_4}}|\cdots|+\sum_{\substack{\lambda_2\ll\lambda_0\approx\lambda_1 \\ \lambda_3\ll\lambda_0\approx\lambda_4}}|\cdots| \\
&=: \mathbf I_{21}+\mathbf I_{22}+\mathbf I_{23}, \\
\mathbf I_3 &\lesssim \sum_{\substack{\lambda_0\ll\lambda_1\approx\lambda_2 \\ \lambda_4\ll\lambda_0\approx\lambda_3}}|\cdots|+\sum_{\substack{\lambda_1\ll\lambda_0\approx\lambda_2 \\ \lambda_4\ll\lambda_0\approx\lambda_3}}|\cdots|+\sum_{\substack{\lambda_2\ll\lambda_0\approx\lambda_1 \\ \lambda_4\ll\lambda_0\approx\lambda_3}}|\cdots| \\
&=: \mathbf I_{31}+\mathbf I_{32}+\mathbf I_{33}. 	
\end{align*}
\subsection{Estimates of $\mathbf I_1$}
If $\lambda_0\ll \lambda_1\approx\lambda_2$, we use in order the H\"older inequality and bilinear estimates \eqref{bi-2d-hh} for the bilinear forms involved in the $(1-\eta)$ exponent and then simply use H\"older inequality for the bilinear forms involved in the $\eta$ exponent. We also use $L^4_{t,x}$-Strichartz estimates Proposition \ref{stri-2d} for $\uppsi_{\Theta_4}$ to obtain the $V^2$ norm. The remaining task is very straightforward as follows:
\begin{align*}
\mathbf I_{11} & \lesssim \sum_{\substack{\lambda_0\ll\lambda_1\approx\lambda_2\\\lambda_0\ll\lambda_3\approx\lambda_4\\\lambda_0 \ge 1}}\lambda_0^{-2}\|P_{\lambda_0}(\varphi_{\Theta_1}^\dagger\beta\phi_{\Theta_2})\|_{L^2_{t,x}}\|P_{\lambda_0}(\uppsi_{\Theta_4}^\dagger\beta\psi_{\Theta_3})\|_{L^2_{t,x}} \\
& \lesssim	\sum_{\substack{\lambda_0\ll\lambda_1\approx\lambda_2\\\lambda_0\ll\lambda_3\approx\lambda_4}}\lambda_0^{-2}\left( \|\varphi_{\Theta_1}\|_{L^4_{t,x}}\|\phi_{\Theta_2}\|_{L^4_{t,x}}\|\psi_{\Theta_3}\|_{L^4_{t,x}}\|\uppsi_{\Theta_4}\|_{L^4_{t,x}} \right)^\eta \\
&\qquad\qquad\qquad\times \left(\|P_{\lambda_0}(\varphi_{\Theta_1}^\dagger\beta\phi_{\Theta_2})\|_{L^2_{t,x}}\|P_{\lambda_0}(\uppsi_{\Theta_4}^\dagger\beta\psi_{\Theta_3})\|_{L^2_{t,x}}\right)^{1-\eta} \\
& \lesssim \sum_{\substack{\lambda_0\ll\lambda_1\approx\lambda_2\\\lambda_0\ll\lambda_3\approx\lambda_4}}\lambda_0^{-2}(\lambda_1\lambda_3)^\eta \left( \lambda_1^{-\frac12}\lambda_2^{-\frac12}\lambda_3^{-\frac12}\lambda_4^{-\frac12}\|\varphi_{\Theta_1}\|_{L^4_{t,x}}\|\phi_{\Theta_2}\|_{L^4_{t,x}}\|\psi_{\Theta_3}\|_{L^4_{t,x}}\|\uppsi_{\Theta_4}\|_{L^4_{t,x}} \right)^\eta \\
& \times \lambda_0^{1-\eta}\left(\frac{\lambda_0}{\lambda_1}\right)^{\frac12(1-\eta)}\left(\frac{\lambda_0}{\lambda_3}\right)^{\frac12(1-\eta)}\left( \|\varphi_{\Theta_1}\|_{V^2_{\theta_1}}\|\phi_{\Theta_2}\|_{V^2_{\theta_2}}\|\psi_{\Theta_3}\|_{V^2_{\theta_3}}\|\uppsi_{\Theta_4}\|_{V^2_{\theta_4}} \right)^{1-\eta} \\
& \lesssim \sum_{\substack{\lambda_0\ll\lambda_1\approx\lambda_2\\\lambda_0\ll\lambda_3\approx\lambda_4}}\lambda_0^{-1+\eta}\left(\frac{\lambda_0}{\lambda_1}\right)^{\frac12(1-3\eta)}\left(\frac{\lambda_0}{\lambda_3}\right)^{\frac12(1-3\eta)}\left( \lambda_1^{-\frac12}\lambda_2^{-\frac12}\lambda_3^{-\frac12}\|\varphi_{\Theta_1}\|_{L^4_{t,x}}\|\phi_{\Theta_2}\|_{L^4_{t,x}}\|\psi_{\Theta_3}\|_{L^4_{t,x}} \right)^\eta \\
&\qquad\qquad\qquad \times \left( \|\varphi_{\Theta_1}\|_{V^2_{\theta_1}}\|\phi_{\Theta_2}\|_{V^2_{\theta_2}}\|\psi_{\Theta_3}\|_{V^2_{\theta_3}} \right)^{1-\eta} \|\uppsi_{\Theta_4}\|_{V^2_{\theta_4}} \\
& \lesssim \sum_{\lambda_1\approx\lambda_2,\lambda_3\approx\lambda_4}\left(\frac{\lambda_{\rm med}}{\lambda_{\max}}\right)^{\frac12(1-3\eta)}\left( \lambda_1^{-\frac12}\lambda_2^{-\frac12}\lambda_3^{-\frac12}\|\varphi_{\Theta_1}\|_{L^4_{t,x}}\|\phi_{\Theta_2}\|_{L^4_{t,x}}\|\psi_{\Theta_3}\|_{L^4_{t,x}} \right)^\eta \\
&\qquad\qquad\qquad \times \left( \|\varphi_{\Theta_1}\|_{V^2_{\theta_1}}\|\phi_{\Theta_2}\|_{V^2_{\theta_2}}\|\psi_{\Theta_3}\|_{V^2_{\theta_3}} \right)^{1-\eta} \|\uppsi_{\Theta_4}\|_{V^2_{\theta_4}}.
\end{align*}
The estimate of $\mathbf I_{12}$ is very similar to $\mathbf I_{11}$. We apply H\"older inequality and then use bilinear estimates \eqref{bi-2d-lh} for the bilinear form $\varphi_{\Theta_1}^\dagger\beta\phi_{\Theta_2}$ and \eqref{bi-2d-hh} for the form $\uppsi_{\Theta_4}^\dagger\beta\psi_{\Theta_3}$ involved in the $(1-\eta)$ exponent, respectively. We again use H\"older inequality for the remaining term involved in $\eta$ exponent.
\begin{align*}
\mathbf I_{12} & \lesssim \sum_{\substack{\lambda_1\ll\lambda_0\approx\lambda_2 \\ \lambda_0\ll\lambda_3\approx\lambda_4\\\lambda_0 \ge 1}}\lambda_0^{-2}(\lambda_1\lambda_2)^\frac\eta2\lambda_3^\eta	\left( \lambda_1^{-\frac12}\lambda_2^{-\frac12}\lambda_3^{-\frac12}\|\varphi_{\Theta_1}\|_{L^4_{t,x}}\|\phi_{\Theta_2}\|_{L^4_{t,x}}\|\psi_{\Theta_3}\|_{L^4_{t,x}} \right)^\eta \\
&\times (\lambda_1\lambda_2)^{\frac12(1-\eta)}\lambda_0^{\frac12(1-\eta)}\left(\frac{\lambda_0}{\lambda_3}\right)^{\frac12(1-\eta)}\left( \|\varphi_{\Theta_1}\|_{V^2_{\theta_1}}\|\phi_{\Theta_2}\|_{V^2_{\theta_2}}\|\psi_{\Theta_3}\|_{V^2_{\theta_3}} \right)^{1-\eta} \|\uppsi_{\Theta_4}\|_{V^2_{\theta_4}} \\
& \lesssim \sum_{\substack{\lambda_1\ll\lambda_0\approx\lambda_2 \\ \lambda_0\ll\lambda_3\approx\lambda_4}}\lambda_0^{-\frac12(1-\eta)}\left(\frac{\lambda_1}{\lambda_0}\right)^\frac12\left(\frac{\lambda_0}{\lambda_3}\right)^{\frac12(1-3\eta)}\left( \lambda_1^{-\frac12}\lambda_2^{-\frac12}\lambda_3^{-\frac12}\|\varphi_{\Theta_1}\|_{L^4_{t,x}}\|\phi_{\Theta_2}\|_{L^4_{t,x}}\|\psi_{\Theta_3}\|_{L^4_{t,x}} \right)^\eta \\
&\qquad\qquad\qquad\times\left( \|\varphi_{\Theta_1}\|_{V^2_{\theta_1}}\|\phi_{\Theta_2}\|_{V^2_{\theta_2}}\|\psi_{\Theta_3}\|_{V^2_{\theta_3}} \right)^{1-\eta} \|\uppsi_{\Theta_4}\|_{V^2_{\theta_4}} \\
& \lesssim \sum_{\lambda_1\ll\lambda_2\ll\lambda_3\approx\lambda_4}\left(\frac{\lambda_2}{\lambda_3}\right)^{\frac12(1-3\eta)}\left( \lambda_1^{-\frac12}\lambda_2^{-\frac12}\lambda_3^{-\frac12}\|\varphi_{\Theta_1}\|_{L^4_{t,x}}\|\phi_{\Theta_2}\|_{L^4_{t,x}}\|\psi_{\Theta_3}\|_{L^4_{t,x}} \right)^\eta \\
&\qquad\qquad\qquad\times\left( \|\varphi_{\Theta_1}\|_{V^2_{\theta_1}}\|\phi_{\Theta_2}\|_{V^2_{\theta_2}}\|\psi_{\Theta_3}\|_{V^2_{\theta_3}} \right)^{1-\eta} \|\uppsi_{\Theta_4}\|_{V^2_{\theta_4}}.
\end{align*}
The estimate of $\mathbf I_{13}$ is identical to the estimate of $\mathbf I_{12}$. The only task is to interchange the input frequencies $\lambda_1$ and $\lambda_2$. In this manner we observe that $\mathbf I_{23}$ and $\mathbf I_{33}$ are identical to $\mathbf I_{22}$ and $\mathbf I_{32}$, respectively. We omit the details.
\subsection{Estimates of $\mathbf I_2$}
The estimate of $\mathbf I_{21}$ is also similar as $\mathbf I_{12}$. We only exchange the role of bilinear forms $\varphi_{\Theta_1}^\dagger\beta\psi_{\Theta_2}$ and $\uppsi_{\Theta_4}^\dagger\beta\psi_{\Theta_3}$.
\begin{align*}
\mathbf I_{21} & \lesssim 	\sum_{\substack{\lambda_0\ll\lambda_1\approx\lambda_2\\ \lambda_3\ll\lambda_0\approx\lambda_4}}\lambda_0^{-2}\lambda_1^\eta (\lambda_3\lambda_4)^\frac\eta2\left( \lambda_1^{-\frac12}\lambda_2^{-\frac12}\lambda_3^{-\frac12}\|\varphi_{\Theta_1}\|_{L^4_{t,x}}\|\phi_{\Theta_2}\|_{L^4_{t,x}}\|\psi_{\Theta_3}\|_{L^4_{t,x}} \right)^\eta\\
&\qquad\times \lambda_0^{\frac12(1-\eta)}\left(\frac{\lambda_0}{\lambda_1}\right)^{\frac12(1-\eta)}(\lambda_3\lambda_4)^{\frac12(1-\eta)}\left( \|\varphi_{\lambda_1}\|_{V^2_{\theta_1}}\|\phi_{\Theta_2}\|_{V^2_{\theta_2}}\|\psi_{\Theta_3}\|_{V^2_{\theta_3}} \right)^{1-\eta} \|\uppsi_{\Theta_4}\|_{V^2_{\theta_4}} \\
& \lesssim \sum_{\substack{\lambda_0\ll\lambda_1\approx\lambda_2\\ \lambda_3\ll\lambda_0\approx\lambda_4}}\lambda_0^{-1-\eta}\lambda_1^{-\frac12(1-3\eta)}(\lambda_3\lambda_4)^\frac12\left( \lambda_1^{-\frac12}\lambda_2^{-\frac12}\lambda_3^{-\frac12}\|\varphi_{\Theta_1}\|_{L^4_{t,x}}\|\phi_{\Theta_2}\|_{L^4_{t,x}}\|\psi_{\Theta_3}\|_{L^4_{t,x}} \right)^\eta \\
&\qquad\qquad\qquad \times \left( \|\varphi_{\Theta_1}\|_{V^2_{\theta_1}}\|\phi_{\Theta_2}\|_{V^2_{\theta_2}}\|\psi_{\Theta_3}\|_{V^2_{\theta_3}} \right)^{1-\eta} \|\uppsi_{\Theta_4}\|_{V^2_{\theta_4}} \\
& \lesssim \sum_{\lambda_3\ll\lambda_4\ll\lambda_1\approx\lambda_2}\lambda_1^{-\frac14}\left(\frac{\lambda_4}{\lambda_1}\right)^\eta\left(\frac{\lambda_3}{\lambda_4}\right)^\frac12\left( \lambda_1^{-\frac12}\lambda_2^{-\frac12}\lambda_3^{-\frac12}\|\varphi_{\Theta_1}\|_{L^4_{t,x}}\|\phi_{\Theta_2}\|_{L^4_{t,x}}\|\psi_{\Theta_3}\|_{L^4_{t,x}} \right)^\eta \\
&\qquad\qquad\qquad \times \left( \|\varphi_{\Theta_1}\|_{V^2_{\theta_1}}\|\phi_{\Theta_2}\|_{V^2_{\theta_2}}\|\psi_{\Theta_3}\|_{V^2_{\theta_3}} \right)^{1-\eta} \|\uppsi_{\Theta_4}\|_{V^2_{\theta_4}}.
\end{align*}
Now we use bilinear estimates \eqref{bi-2d-lh} and get
\begin{align*}
\mathbf I_{22} & \lesssim \sum_{\substack{\lambda_1\ll\lambda_0\approx\lambda_2\\ \lambda_3\ll\lambda_0\approx\lambda_4}}\lambda_0^{-2}(\lambda_1\lambda_2\lambda_3\lambda_4)^\frac12	\left( \lambda_1^{-\frac12}\lambda_2^{-\frac12}\lambda_3^{-\frac12}\|\varphi_{\Theta_1}\|_{L^4_{t,x}}\|\phi_{\Theta_2}\|_{L^4_{t,x}}\|\psi_{\Theta_3}\|_{L^4_{t,x}} \right)^\eta \\
&\qquad\qquad\qquad \times \left( \|\varphi_{\Theta_1}\|_{V^2_{\theta_1}}\|\phi_{\Theta_2}\|_{V^2_{\theta_2}}\|\psi_{\Theta_3}\|_{V^2_{\theta_3}} \right)^{1-\eta} \|\uppsi_{\Theta_4}\|_{V^2_{\theta_4}} \\
& \lesssim \sum_{\lambda_1,\lambda_3\ll\lambda_2\approx\lambda_4}\left(\frac{\lambda_1\lambda_3}{\lambda_2\lambda_4}\right)^\frac12\left( \lambda_1^{-\frac12}\lambda_2^{-\frac12}\lambda_3^{-\frac12}\|\varphi_{\Theta_1}\|_{L^4_{t,x}}\|\phi_{\Theta_2}\|_{L^4_{t,x}}\|\psi_{\Theta_3}\|_{L^4_{t,x}} \right)^\eta \\
&\qquad\qquad\qquad \times \left( \|\varphi_{\Theta_1}\|_{V^2_{\theta_1}}\|\phi_{\Theta_2}\|_{V^2_{\theta_2}}\|\psi_{\Theta_3}\|_{V^2_{\theta_3}} \right)^{1-\eta} \|\uppsi_{\Theta_4}\|_{V^2_{\theta_4}}.
\end{align*}
\subsection{Estimates of $\mathbf I_3$}
The estimate of $\mathbf I_{31}$ is similar to $\mathbf I_{21}$. The only difference is to exchange the role of $\psi_{\Theta_3}$ and $\uppsi_{\Theta_4}$. Indeed, we have
\begin{align*}
\mathbf I_{31} & \lesssim \sum_{\substack{\lambda_0\ll\lambda_1\approx\lambda_2 \\ \lambda_4\ll \lambda_0\approx\lambda_3}}\lambda_0^{-2}\lambda_1^\eta \lambda_0^{\frac12(1-\eta)}\left(\frac{\lambda_0}{\lambda_1}\right)^{\frac12(1-\eta)} (\lambda_3\lambda_4)^\frac12	 \left( \lambda_1^{-\frac12}\lambda_2^{-\frac12}\lambda_3^{-\frac12}\|\varphi_{\Theta_1}\|_{L^4_{t,x}}\|\phi_{\Theta_2}\|_{L^4_{t,x}}\|\psi_{\Theta_3}\|_{L^4_{t,x}} \right)^\eta \\
&\qquad\qquad\qquad \times \left( \|\varphi_{\Theta_1}\|_{V^2_{\theta_1}}\|\phi_{\Theta_2}\|_{V^2_{\theta_2}}\|\psi_{\Theta_3}\|_{V^2_{\theta_3}} \right)^{1-\eta} \|\uppsi_{\Theta_4}\|_{V^2_{\theta_4}}.
\end{align*}
The simple use of bilinear estimates \eqref{bi-2d-lh} yields
\begin{align*}
\mathbf I_{32} & \lesssim \sum_{\substack{\lambda_1\ll\lambda_0\approx\lambda_2 \\ \lambda_4\ll \lambda_0\approx\lambda_3}}	\lambda_0^{-2}(\lambda_1\lambda_2\lambda_3\lambda_4)^\frac12	\left( \lambda_1^{-\frac12}\lambda_2^{-\frac12}\lambda_3^{-\frac12}\|\varphi_{\Theta_1}\|_{L^4_{t,x}}\|\phi_{\Theta_2}\|_{L^4_{t,x}}\|\psi_{\Theta_3}\|_{L^4_{t,x}} \right)^\eta \\
&\qquad\qquad\qquad \times \left( \|\varphi_{\Theta_1}\|_{V^2_{\theta_1}}\|\phi_{\Theta_2}\|_{V^2_{\theta_2}}\|\psi_{\Theta_3}\|_{V^2_{\theta_3}} \right)^{1-\eta} \|\uppsi_{\Theta_4}\|_{V^2_{\theta_4}} \\
& \lesssim \sum_{\lambda_1,\lambda_4\ll\lambda_2\approx\lambda_3}\left(\frac{\lambda_1\lambda_4}{\lambda_2\lambda_3}\right)^\frac12\left( \lambda_1^{-\frac12}\lambda_2^{-\frac12}\lambda_3^{-\frac12}\|\varphi_{\Theta_1}\|_{L^4_{t,x}}\|\phi_{\Theta_2}\|_{L^4_{t,x}}\|\psi_{\Theta_3}\|_{L^4_{t,x}} \right)^\eta \\
&\qquad\qquad\qquad \times \left( \|\varphi_{\Theta_1}\|_{V^2_{\theta_1}}\|\phi_{\Theta_2}\|_{V^2_{\theta_2}}\|\psi_{\Theta_3}\|_{V^2_{\theta_3}} \right)^{1-\eta} \|\uppsi_{\Theta_4}\|_{V^2_{\theta_4}}.
\end{align*}
Finally, by the estimates of $\mathbf I_j$, $j=1,2,3$ we conclude that
\begin{align}
\begin{aligned}
\mathfrak I_{\lambda\gg h}&\lesssim \sum_{\lambda_4\ge1}\bigg( \sum_{\lambda_1,\lambda_2,\lambda_3\ge1}\left(\frac{\lambda_{\rm med}}{\lambda_{\max}}\right)^\delta \left( \lambda_1^{-\frac12}\lambda_2^{-\frac12}\lambda_3^{-\frac12}\|\varphi_{\Theta_1}\|_{L^4_{t,x}}\|\phi_{\Theta_2}\|_{L^4_{t,x}}\|\psi_{\Theta_3}\|_{L^4_{t,x}} \right)^\eta \\
&\qquad\qquad\qquad\qquad\qquad\qquad\left( \|\varphi_{\Theta_1}\|_{V^2_{\theta_1}}\|\phi_{\Theta_2}\|_{V^2_{\theta_2}}\|\psi_{\Theta_3}\|_{V^2_{\theta_3}} \right)^{1-\eta}  \bigg)^2.	
\end{aligned}
\end{align}
Thus we have
\begin{align*}
	\|\mathcal I_{\theta}[V_b*(\varphi^\dagger\beta\phi)\beta\psi](t_0; \cdot)\|_{F^{0,0}(I)} & \le C (\|\varphi\|_{\mathbf D^{-\frac12, 0}(I)}\|\phi\|_{\mathbf D^{-\frac12, 0}(I)}\|\psi\|_{\mathbf D^{-\frac12, 0}(I)})^\eta \\
	&\qquad\qquad \times (\|\varphi\|_{F^{0,0}(I)}\|\phi\|_{F^{0,0}(I)}\|\psi\|_{F^{0,0}(I)})^{1-\eta},
\end{align*}	
which is the desired estimate.
\section{Majorana condition: Proof of Theorem \ref{majorana-gwp}}\label{sec-maj}
We define the \textit{charge conjugation operator} $C^z_\theta$ by
$$
C^z_\theta\psi = \theta z\gamma^2\psi^*,
$$
where $z = e^{i\omega}$ for some $\omega \in \mathbb R$ and $\theta \in \{+, -\}$.
We also define the projection operator $P_\theta^z$ by
$$
P^z_\theta\psi = \frac12\left(\psi+C^z_\theta\psi\right).
$$
Then we get 
\[
P_+^z + P_-^z = I.
\]
\begin{prop}
Let $z=e^{i\omega}$ for any $\omega\in\mathbb R$ and $\theta\in\{+,-\}$. The operator $P^z_\theta$ satisfies the following:
\begin{align}\label{eq:proj-z}
(P^z_\theta)^2 = P^z_\theta, \quad P^z_\theta P^z_{-\theta} = 0.	
\end{align}
\end{prop}
\begin{proof}
An simple computation gives $P^z_+ + P^z_{-}$ is the identity operator. One can observe that
\begin{align*}
P^z_\theta C^z_\theta\psi &= P^z_\theta(\theta z\gamma^2\psi^*)  = \frac12(\theta z\gamma^2\psi^*+\theta z\gamma^2\theta z^*(\gamma^2)^*\psi)  = \frac12(\theta z\gamma^2\psi^*+\psi) = P^z_\theta \psi. 	
\end{align*}
Here we used $(\gamma^2)^\dagger = -\gamma^2$ and $\gamma^2\times(-\gamma^2) = \mathbb I_{\tilde d}$. Then we see that
\begin{align*}
P^z_\theta P^z_\theta\psi & = \frac12 P^z_\theta(\psi+C^z_\theta\psi) = \frac12(P^z_\theta\psi+P^z_\theta C^z_\theta\psi)  = \frac12(P^z_\theta\psi + P^z_\theta\psi) = P^z_\theta\psi,	
\end{align*}
and hence we conclude that $(P^z_\theta)^2\psi=P^z_\theta\psi$. This implies that $P^z_\theta P^z_{-\theta} = 0$, since
\begin{align*}
P^z_\theta P^z_{-\theta}\psi &= P^z_\theta(I-P^z_\theta)\psi = P^z_\theta\psi - (P^z_\theta)^2\psi  = P^z_\theta\psi-P^z_\theta\psi = 0.
\end{align*}
This finishes the proof of \eqref{eq:proj-z}.
\end{proof}
Given a spinor field $\psi:\mathbb R^{1+d}\rightarrow\mathbb C^{\tilde d}$, we have the decomposition
$$
\psi = \sum_{\theta \in \{+, -\}}P^z_\theta \psi.
$$
The following proposition is readily obtained. However it presents the significant observation on the study of large data well-posedness via the Majorana condition.
\begin{prop}
For any spinor field $\psi:\mathbb R^{1+d}\rightarrow \mathbb C^{\tilde d}$, recall that we define $\overline\psi = \psi^\dagger\gamma^0$. Then we have the following identity:	$$
\overline{P^z_\theta\psi}P^z_\theta\psi=0
$$
\end{prop}
\begin{proof}
To see this, we first note that $\overline{P^z_\theta\psi}P^z_\theta\psi$ must be real-valued. We write
\begin{align*}
\overline{P^z_\theta\psi}P^z_\theta\psi&= \frac14(\psi^\dagger+\theta z^*\psi^T(-\gamma^2))\gamma^0(\psi+\theta z\gamma^2\psi^*),	
\end{align*}
where $\psi^T$ is the transpose of spinor field $\psi$. Then we obtain
\begin{align*}
	\overline{P^z_\theta\psi}P^z_\theta\psi&= \frac14\left(\overline\psi\psi+\theta z\psi^\dagger\gamma^0\gamma^2\psi^*-\theta z^*\psi^T\gamma^2\gamma^0\psi+\psi^T(-\gamma^2)\gamma^0\gamma^2\psi^*\right).
\end{align*}
Here, we observe that
\begin{align*}
\psi^T(-\gamma^2)\gamma^0\gamma^2\psi^* &= \psi^T\gamma^0(\gamma^2)^2\psi^* = -\psi^T\gamma^0\psi^* = -(\psi^\dagger\gamma^0\psi)^* = -(\overline\psi\psi)^* = -\overline\psi\psi.	
\end{align*}
Similarly, we see that
\begin{align*}
-\theta z^*\psi^T\gamma^2\gamma^0\psi &= \theta z^*\psi^T\gamma^0\gamma^2\psi = -(\theta z\psi^\dagger\gamma^0\gamma^2\psi^*)^*.	
\end{align*}
Combining these identities, we deduce that
\begin{align*}
	\overline{P^z_\theta\psi}P^z_\theta\psi &= \frac i2\textrm{Im}(\theta z\overline\psi\gamma^2\psi),
\end{align*}
which should be zero, since the LHS is purely real, whereas the RHS is purely imaginary.
\end{proof}
Now we consider the initial value problems for the Dirac equations
\begin{align}\label{eq:c-dirac}
\left\{
\begin{array}{l}
	-i\gamma^\mu\partial_\mu\psi+\psi = [V_b*(\psi^\dagger \beta\psi)]\psi, \\
	\psi|_{t=0}=\psi_0.
\end{array}
\right.	
\end{align}
Recall that we have put the mass parameter $M=1$.
We shall study the time evolution property of solutions to the equations for large data by exploiting the Majorana condition. The first step is to consider the system of cubic Dirac equations instead of the above equation which presents
\begin{align}\label{psi-decomp}
\begin{aligned}
-i\gamma^\mu\partial_\mu\varphi+\varphi = V_b*(\overline{P^z_\theta\varphi}P^z_{-\theta}\phi+\overline{P^z_{-\theta}\phi}P^z_\theta\varphi)\varphi, \\
-i\gamma^\mu\partial_\mu\phi+\phi = V_b*(\overline{P^z_\theta\varphi}P^z_{-\theta}\phi+\overline{P^z_{-\theta}\phi}P^z_\theta\varphi)\phi.
\end{aligned}
\end{align}
for sufficiently smooth $\varphi, \phi$ with initial data
\begin{align}\label{decomp-data}
\varphi|_{t=0} = P^z_\theta\psi_0,\quad \phi|_{t=0} = P^z_{-\theta}\psi_0.	
\end{align}
\noindent The aim of this section is to prove the following.
\begin{thm}\label{majorana-gwp1}
Let $z = e^{i\omega}, \omega \in \mathbb R$. Let $\sigma > 0$ for $d = 3$ and $\sigma = 0$ for $d = 2$. Then there exists $0 < \epsilon < 1$ such that for any ${\mathsf A} \ge 1$ and any $0 < \mathtt a \le  \epsilon {\mathsf A}^{-1}$, if the initial data satisfy
$$
\|P^z_+\psi_0\|_{L^{2,\sigma}(\mathbb R^d)} \le \mathtt a,\ \|P^z_{-}\psi_0\|_{L^{2,\sigma}(\mathbb R^d)}\le {\mathsf A},
$$
then the equation \eqref{dirac} is globally well-posed. To be precise, the Cauchy problem of \eqref{psi-decomp} is globally well-posed in the sense that
\begin{align}
P^z_+\psi,P^z_{-}\psi \in C(\mathbb R;L^{2,\sigma}(\mathbb R^d))	.
\end{align}
Furthermore, there exist $\varphi_0^{\pm},\phi_{0}^{\pm}\in L^{2,\sigma}$ such that
\begin{align*}
\lim_{t\rightarrow\pm\infty}\|P^z_+\psi(t) - e^{\theta it\Lambda}\varphi_{0}^{\pm}\|_{L^{2,\sigma}(\mathbb R^d)} = 0, \\
\lim_{t\rightarrow\pm\infty}\|P^z_{-}\psi(t) - e^{\theta it\Lambda}\phi_{0}^{\pm}\|_{L^{2,\sigma}(\mathbb R^d)} = 0.
\end{align*}
\end{thm}
\begin{proof}[Proof of Theorem \ref{majorana-gwp1}]
We recall the Banach space $F^{0,\sigma}\subset C(\mathbb R;L^{2,\sigma}(\mathbb R^{d}))$. If $\psi\in F^{0,\sigma}$ is a solution to the cubic Dirac equations \eqref{eq:c-dirac},
then by the multilinear estimates given by  putting $\eta=0$ in the multilinear estimates Theorem \ref{main-tri-est}, we have
\begin{align}\label{psi-bdd-maj}
\|\psi\|_{F^{0,\sigma}} &\le \|\psi_0\|_{L^{2,\sigma}}+C\|\psi_1\|_{F^{0,\sigma}}\|\psi_2\|_{F^{0,\sigma}}\|\psi_3\|_{F^{0,\sigma}}.
\end{align}
 We consider the set
 $$
 X = \{ (\varphi,\phi)\in F^{0,\sigma}\times F^{0,\sigma} : \|\varphi\|_{F^{0,\sigma}}\le 2\|\varphi(0)\|_{L^{2,\sigma}},\ \|\phi\|_{F^{0,\sigma}}\le2\|\phi(0)\|_{L^{2,\sigma}} \}
 $$
 and for ${\mathsf A}, \mathtt a > 0$, we define the norm
 $$
 \|(\varphi,\phi)\|_{X} = \mathtt a^{-1}\|\varphi\|_{F^{0,\sigma}} + {\mathsf A}^{-1}\|\phi\|_{F^{0,\sigma}}.
 $$
 Then $X$ is a complete metric space with the metric corresponding to the norm. Now we let $\Phi=(\Phi_1,\Phi_2)$ be the inhomogeneous solution map for \eqref{psi-decomp} given by the Duhamel principle. Then the bound \eqref{psi-bdd-maj} and the definition of the set $X$ give
 \begin{align}
 \begin{aligned}
 \|\Phi_1(\varphi,\phi)\|_{F^{0,\sigma}} & \le \|\varphi(0)\|_{L^{2,\sigma}}+C \|\varphi\|_{F^{0,\sigma}}^2\|\phi\|_{F^{0,\sigma}} \\
 & \le \|\varphi(0)\|_{L^{2,\sigma}}	+16C \|\varphi(0)\|_{L^{2,\sigma}}^2\|\phi(0)\|_{L^{2,\sigma}}, \\
 & \le 	(1+16C{\mathsf A}\mathtt a)\|\varphi(0)\|_{L^{2,\sigma}},
 \end{aligned}
 \end{align}
and also
\begin{align}
\begin{aligned}
	\|\Phi_2(\varphi,\phi)\|_{F^{0,\sigma}} & \le \|\phi(0)\|_{L^{2,\sigma}}+C \|\phi\|_{F^{0,\sigma}}^2\|\varphi\|_{F^{0,\sigma}} \\
 & \le \|\phi(0)\|_{L^{2,\sigma}}	+16C \|\phi(0)\|_{L^{2,\sigma}}^2\|\varphi(0)\|_{L^{2,\sigma}}, \\
 & \le 	(1+16C{\mathsf A}\mathtt a)\|\phi(0)\|_{L^{2,\sigma}}.
\end{aligned}	
\end{align}
Then we put $\mathtt a \le \dfrac{1}{16C{\mathsf A}}$ and deduce that the map $\Phi$ is the flow map from $X$ into $X$. By multilinear estimates lead us that the map $\Phi$ is a contraction on the set $X$. Indeed, suppose that we have $(\varphi_1,\phi_1),\,(\varphi_2,\phi_2)\in X$. Then we estimate
$$
\|\Phi_1(\varphi_1,\phi_1)-\Phi_1(\varphi_2,\phi_2)\|_{F^{0, \sigma}} \le 16C{\mathsf A}\mathtt a \|\varphi_1-\varphi_2\|_{F^{0,\sigma}} + 8C\mathtt a^2\|\phi_1-\phi_2\|_{F^{0,\sigma}}
$$
and
$$
\|\Phi_2(\varphi_1,\phi_1)-\Phi_2(\varphi_2,\phi_2)\|_{F^{0,\sigma}} \le 16C{\mathsf A}\mathtt a \|\phi_1-\phi_2\|_{F^{0,\sigma}} + 8C{\mathsf A}^{2}\|\varphi_1-\varphi_2\|_{F^{0,\sigma}}.
$$
In consequence we obtain
\begin{align*}
\|\Phi(\varphi_1,\phi_1)-\Phi(\varphi_2,\phi_2)\|_X &\le 24C{\mathsf A}\|\varphi_1 -\varphi_2\|_{F^{0,\sigma}} + 24C\mathtt a\|\phi_1-\phi_2\|_{F^{0,\sigma}}\\
& = 24C{\mathsf A}\mathtt a\|(\varphi_1,\phi_1)-(\varphi_2,\phi_2)\|_X.
\end{align*}
Thus by choosing $\epsilon = \dfrac{1}{48C}$, the solution map $\Phi$ is a contraction on $X$ for any $\mathtt a \le \epsilon{\mathsf A}^{-1}$.
\end{proof}

\section*{Acknowledgements}
This work was supported in part by NRF-2021R1I1A3A04035040(Republic of Korea).


\end{document}